\documentclass[11pt,leqno]{article}

\usepackage{amsthm,amsfonts,amssymb,amsmath,oldgerm}
\usepackage{fullpage}
\usepackage{graphicx}
\usepackage{mathrsfs}
\usepackage[dvipsnames]{xcolor}
\usepackage{xcolor}

\numberwithin{equation}{section}

\renewcommand\d{\partial}
\renewcommand\a{\alpha}
\renewcommand\b{\beta}

\def\eps {\varepsilon}

\newcommand{\Z}{\mathbb Z}
\newcommand{\R}{\mathbb R}

\newcommand\LA{\left\langle}
\newcommand\RA{\right\rangle}

\newcommand{\RM}{{\mathbb{R}}}
\newcommand{\CM}{{\mathbb{C}}}
\newcommand{\NM}{{\mathbb{N}}}
\newcommand{\ZM}{{\mathbb{Z}}}
\newtheorem{theorem}{Theorem}[section]
\newtheorem{proposition}[theorem]{Proposition}
\newtheorem{corollary}[theorem]{Corollary}
\newtheorem{lemma}[theorem]{Lemma}
\newtheorem{remark}[theorem]{Remark}
\theoremstyle{definition}
\newtheorem{definition}[theorem]{Definition}

\allowdisplaybreaks[3]

\title{Linear Modulational and Subharmonic Dynamics of Spectrally Stable Lugiato-Lefever Periodic Waves}

\author{Mariana Haragus\thanks{FEMTO-ST institute, Univ. Bourgogne-Franche Comt\'e, 15b avenue des Montboucons, 25030 Besan\c con cedex, France; mharagus@univ-fcomte.fr},\quad Mathew~A.~Johnson\thanks{Department of Mathematics, University of Kansas, 1460 Jayhawk Boulevard, 
Lawrence, KS 66045, USA; matjohn@ku.edu}\quad\&\quad Wesley R. Perkins\thanks{Department of Mathematics, University of Kansas, 1460 Jayhawk Boulevard, 
Lawrence, KS 66045, USA; wesley.perkins@ku.edu} }

\date{\today}

\begin{document}

\maketitle

\begin{abstract}
We study the linear dynamics of spectrally stable $T$-periodic stationary solutions of the Lugiato-Lefever equation (LLE), a damped  nonlinear Schr\"odinger equation with forcing that arises
in nonlinear optics.  
Such $T$-periodic solutions are nonlinearly stable to $NT$-periodic, i.e. subharmonic, perturbations for each $N\in\NM$ with exponential decay rates of perturbations of the form $e^{-\delta_N t}$.  
However, both the exponential rates of decay $\delta_N$ and the allowable size of the initial perturbations
tend to $0$ as $N\to\infty$, so that this result is non-uniform in $N$ and, in fact, empty in the limit $N=\infty$. The primary goal
of this paper is to introduce a methodology, in the context of the LLE, by which a uniform stability result for subharmonic perturbations
may be achieved, at least at the linear level.  The obtained uniform decay rates are shown to agree precisely with the polynomial decay rates of localized, i.e. integrable on the real line,
perturbations of such spectrally stable periodic solutions of the LLE.  This work both unifies and expands on several existing works in the literature concerning the
stability and dynamics of such waves, and sets forth a general methodology for studying such problems in other contexts.
\end{abstract}

\section{Introduction}\label{S:intro}

In this paper, we consider the stability and dynamics of periodic stationary solutions of the Lugiato-Lefever equation (LLE)
\begin{equation}\label{e:LLE}
\psi_t = -i\b \psi_{xx} - (1+i\a)\psi + i|\psi|^2\psi + F,
\end{equation}
where $\psi(x,t)$ is a complex-valued function depending on a temporal variable $t$ and a spatial variable $x$, the parameters $\alpha,\beta$ are real, 
and $F$ is a positive constant.   The model \eqref{e:LLE} was derived from Maxwell's equations in \cite{LL87} as a model to study pattern formation within the optical field 
in a dissipative and nonlinear optical cavity filled with a Kerr medium and subjected to a continuous laser pump.  In this context, 
$\psi(x,t)$ represents the field envelope, $\alpha>0$ represents a detuning parameter, $F>0$ represents a normalized pump strength,
and $|\beta|=1$ is the dispersion parameter.  The case $\beta=1$ is referred to as the ``normal" dispersion case while $\beta=-1$ is referred to as the ``anomalous" dispersion case.  

Since its derivation, the LLE has been intensely studied in the physics literature in the context of nonlinear optics, having more recently become
a model for high-frequency optical combs generated by microresonators in periodic optical wave guides (see, for example,
\cite{CGTM17} and references therein).  
Until recently, however, there have been relatively few mathematically rigorous studies of the LLE.  
Several recent works have established the existence of periodic standing solutions of \eqref{e:LLE}. 
Such solutions $\psi(x,t)=\phi(x)$ correspond to $T$-periodic solutions of the profile equation
\begin{equation}\label{e:profile}
-i\b \phi'' - (1+i\a)\phi + i|\phi|^2\phi + F=0.
\end{equation}
Using tools from bifurcation theory, the existence of periodic standing waves bifurcating both locally and globally from constant solutions has been shown in \cite{MOT1,MR17,DH18_1,DH18_2}. Another type of periodic solutions has been recently constructed in the case of anomalous dispersion $\beta=-1$ in \cite{HSS19}.  
These solutions correspond to bifurcations from the standard arbitrary amplitude dnoidal solutions of the cubic NLS equation, for small $|(F,\alpha)|$. We also refer to \cite{Go17} for a local bifurcation analysis of bounded solutions of \eqref{e:profile} including, besides periodic, also localized and quasi-periodic solutions.

Our work focuses on the dynamical stability and long-time asymptotic dynamics of spatially periodic standing solutions of \eqref{e:LLE} when subject to varying classes of perturbations.  
Note that if $\phi$ is a $T$-periodic standing solution of \eqref{e:LLE} and we decompose $\phi=\phi_r+i\phi_i$ into its real and imaginary parts, then a function of the form $\psi(x,t)=\phi(x)+v(x,t)$, with $v=v_r+iv_i$, is a solution of \eqref{e:LLE} provided it 
satisfies the real system
\begin{equation}\label{e:lin}
\partial_t\left(\begin{array}{c}v_r\\v_i\end{array}\right)=\mathcal{A}[\phi]\left(\begin{array}{c}v_r\\v_i\end{array}\right)+\mathcal{N}(v),
\end{equation}
where here $\mathcal{N}(v)$ is at least quadratic in $v$ and $\mathcal A[\phi]$ is the (real) linear differential operator
\begin{equation}\label{e:Aphi}
\mathcal A[\phi]=- I+\mathcal{J}\mathcal{L}[\phi],
\end{equation}
with
\[
\mathcal{J}=\left(\begin{array}{cc}0&-1\\1&0\end{array}\right),\quad 
\mathcal{L}[\phi] = \left(\begin{array}{cc} -\b \d_x^2 - \a  + 3\phi_{r}^2 + \phi_{i}^2 & 2\phi_{r}\phi_{i} \\
  2\phi_{r}\phi_{i} & -\b \d_x^2 - \a  + \phi_{r}^2 + 3\phi_{i}^2\end{array}\right).
  \]
The choice of a function space  for the evolution problem \eqref{e:lin} is determined by the class of perturbations of the $T$-periodic standing wave. Choosing a Hilbertian framework, we take $L^2_{\rm per}(0,T)$ for co-periodic perturbations, $L^2_{\rm per}(0,NT)$ with $N\in\NM$ for so-called subharmonic perturbations, and $L^2(\RM)$ for localized perturbations.\footnote{Since operators here are defined on vector valued functions, throughout this work we will abuse notation slightly and write $L^2(\RM)$ rather than $L^2(\RM)\times L^2(\RM)$, and similarly for all other Lebesgue and Sobolev spaces. Furthermore, when the meaning is clear from context, we will write functions in $(f_1,f_2)\in L^2\times L^2$ as simply $f\in L^2$.}
Recall that the spectral stability of a periodic wave $\phi$ to a given class of perturbations is determined by the spectrum of the linear operator $\mathcal{A}[\phi]$ when acting
on the associated function space.  Similarly, 
the linear stability of $\phi$ is given by the properties of the associated evolution semigroup $(e^{\mathcal{A}[\phi]t})_{t\geq 0}$, and the nonlinear (orbital) 
stability by the behavior of the solutions of the nonlinear equation \eqref{e:lin}.

The spectral stability of periodic waves bifurcating locally from constant solutions has been studied in \cite{DH18_2,DH18_1}. It turns out that most of these waves are unstable for subharmonic perturbations, with $N$ larger than a certain value $N_c\geq1$ depending on the parameters $\alpha$ and $F$, and that there is precisely one family of such waves which are 
spectrally stable for all subharmonic perturbations, and also for localized perturbations. These stable periodic waves bifurcate supercritically in the case of anomalous dispersion, $\beta=-1$, for any fixed parameter $\alpha<41/30$ and bifurcation parameter $F^2=F_1^2+\mu$, for sufficiently small $\mu>0$, where $F_1^2=(1-\alpha)^2+1$. More precisely, there exists $\mu_0>0$ such that for any $\mu\in(0,\mu_0)$, the LLE has an even periodic solution with Taylor expansion
\begin{equation}\label{e:periodic}
\phi_\mu(x) = \phi^* +  \frac{3(\alpha+i(2-\alpha))}{F_1(41-30\alpha)^{1/2}}\,\cos(\sqrt{2-\alpha}\,x)\,\mu^{1/2}+ O(\mu),
\end{equation}
where $\phi^*$ is the unique constant solution satisfying the algebraic equation
\[
(1+i\alpha)\phi-i\phi|\phi|^2=F_1.
\]
The solution $\phi_\mu$ is $T$-periodic with period $T=2\pi/\sqrt{2-\alpha}$, and it is of class $C^\infty$. 
We point out that this parameter regime has been investigated in the original work of Lugiato and Lefever \cite{LL87} who determined the value $\alpha_c=41/30$ as an instability threshold. 

For co-periodic perturbations which are $H^2$, i.e., belong to the domain of the linear operator $\mathcal A[\phi]$, the nonlinear asymptotic stability of the periodic waves \eqref{e:periodic} is a direct consequence of the bifurcation analysis used for their construction \cite{MOT1, DH18_1}. 
Using Strichartz-type estimates, this result has been extended to more general 
$L^2_{\rm per}(0,T)$-perturbations in \cite{MOT2}. As pointed out in \cite[Section 6(a)]{DH18_1}, the bifurcation analysis used to construct these periodic waves can be extended to spaces of $NT$-periodic functions, for any arbitrary but fixed $N$, which then also gives a nonlinear stability result for these waves for $H^2$-subharmonic perturbations. 
However, this stability result is not uniform in $N$: for a given periodic wave $\phi_\mu$ as in \eqref{e:periodic}, nonlinear stability is obtained for a finite number of integers $N$, 
only.\footnote{This is a result of their construction, which is based on center manifold techniques.}

For  localized perturbations, as well as for general bounded perturbations, the spectral stability of the periodic waves \eqref{e:periodic} has been proved in \cite[Theorem 4.3]{DH18_1}. Based on Floquet-Bloch theory and spectral perturbation theory, this result  shows that the periodic waves given by  \eqref{e:periodic} are \emph{diffusively spectrally stable} in the sense of the following definition.

\begin{definition}\label{Def:spec_stab}
A $T$-periodic stationary solution $\phi\in H^1_{\rm loc}(\RM)$ of \eqref{e:LLE} is said to be \emph{diffusively spectrally stable} provided the following conditions hold:
\begin{enumerate}
\item the spectrum of the linear operator $\mathcal{A}[\phi]$ given by \eqref{e:Aphi} and acting in $L^2(\R)$ satisfies \[\sigma_{L^2(\R)}(\mathcal{A}[\phi])\subset\{\lambda\in\CM:\Re(\lambda)<0\}\cup\{0\};\]
\item there exists $\theta>0$ such that for any $\xi\in[-\pi/T,\pi/T)$ the real part of the spectrum of the Bloch operator $\mathcal{A}_\xi[\phi]:=e^{-i\xi x}\mathcal{A}[\phi]e^{i\xi x}$ acting on 
$L^2_{\rm per}(0,T)$ satisfies
  \[\Re\left(\sigma_{L^2_{\rm per}(0,T)}(\mathcal{A}_\xi[\phi])\right)\leq-\theta \xi^2;\]
\item $\lambda=0$ is a simple eigenvalue of $\mathcal{A}_0[\phi]$ with associated eigenvector the derivative $\phi'$ of the periodic wave.\footnote{Recall, by our slight abuse of notation above, here
we are writing $\phi'$ instead of the technically correct $(\phi_r',\phi_i')$.}
\end{enumerate}
\end{definition}

This stability notion was first introduced in \cite{JNRZ_Invent} for more general classes of viscous conservation and balance laws,
and is stated here in the context of \eqref{e:LLE}. The properties (i)-(iii) in this definition are the main assumptions required for the present analysis.
Recall that Floquet-Bloch theory shows that the spectrum of $\mathcal{A}[\phi]$ acting in $L^2(\R)$
is equal to the union of the spectra of the Bloch operators  $\mathcal{A}_\xi[\phi]$ acting in $L_{\rm per}^2(0,T)$ for  $\xi\in[-\pi/T,\pi/T)$. For subharmonic perturbations, the operator $\mathcal{A}[\phi]$ acts in $L_{\rm per}^2(0,NT)$, and its spectrum is the union of the spectra of the Bloch operators  $\mathcal{A}_\xi[\phi]$ acting in $L_{\rm per}^2(0,T)$ for  $\xi$ in a discrete subset of the interval $[-\pi/T,\pi/T)$ such that $e^{i\xi NT}=1$. In particular, this implies that diffusively spectrally stable periodic waves are spectrally stable for all subharmonic perturbations. Notice that for such perturbations, the spectrum of  $\mathcal{A}[\phi]$ is purely point spectrum consisting of isolated eigenvalues with finite algebraic multiplicities,  $\lambda=0$ is a simple eigenvalue, with associated eigenvector the derivative $\phi'$ of the periodic wave, and the remaining eigenvalues have negative real parts, satisfying the spectral gap condition
\begin{equation}\label{e:deltaN}
\Re\left(\sigma_{L^2_{\rm per}(0,NT)}(\mathcal{A}[\phi])\setminus\{0\}\right)\leq -\delta_N,
\end{equation}
for some $\delta_N>0$.
As the eigenvalues of the Bloch operators $\mathcal{A}_\xi[\phi]$ depend continuously on $\xi$, it is not difficult to see that the  spectral gap $\delta_N$ above tends to $0$, as $N\to\infty$. 
For more information, see Sections \ref{S:bloch_loc}-\ref{S:bloch_per} below.

For the periodic waves bifurcating from the dnoidal solutions of the NLS equation, the authors proved in \cite{HSS19} that some of these are spectrally stable to co-periodic perturbations. 
That this spectral stability corresponds to nonlinear stability was recently established in \cite[Theorem 1]{SS19}. The power of this result is that it reduces the problem of nonlinear stability for co-periodic perturbations to a spectral problem which, in turn, may be amenable 
to well-conditioned analytical and numerical methods. Furthermore, the proof of this result can be easily extended to subharmonic perturbations, leading to the following nonlinear stability result for diffusively spectrally stable periodic standing solutions of \eqref{e:LLE}. 

\begin{theorem}\label{T:SS}
Let $\phi\in H^1_{\rm loc}(\RM)$ be a $T$-periodic standing solution of \eqref{e:LLE} and fix $N\in\NM$.  Assume that $\phi$ is diffusively spectrally stable in the 
sense of Definition~\ref{Def:spec_stab} and, for each $N\in\NM$, take $\delta_N>0$ such that \eqref{e:deltaN} holds.
Then for each $N\in\NM$, $\phi$ is asymptotically stable to subharmonic $NT$-periodic perturbations.  More precisely, for every 
$\delta\in(0,\delta_N)$ there exists an $\eps=\eps_\delta>0$
and a constant $C=C_\delta>0$ such that whenever $u_0\in H^1_{\rm per}(0,NT)$ and $\|u_0-\phi\|_{H^1(0,NT)}<\eps$, then the solution 
$u$ of \eqref{e:LLE} with initial data $u(0)=u_0$ exists globally in time and satisfies
\[
\|u(\cdot,t)-\phi(\cdot-\gamma_\infty)\|_{H^1(0,NT)}\leq C e^{-\delta t}\|u_0-\phi\|_{H^1(0,NT)},
\]
for all $t>0$, where $\gamma_\infty=\gamma_\infty(N)$ is some real constant.
\end{theorem}

The key to the proof of Theorem \ref{T:SS}, which is presented in \cite{SS19} for $N=1$, is a careful estimate of the resolvent operator, which allows one to apply the Gearhart-Pr\"uss theorem and obtain an exponential decay rate for the semigroup generated by the linear operator $\mathcal{A}[\phi]$.  Specifically, one shows that for each $\delta\in(0,\delta_N)$ there exists a constant $C=C_\delta>0$
such that
\begin{equation}\label{e:SS_lin_est}
\left\|e^{\mathcal{A}[\phi]t}\left(1-\mathcal{P}_{0,N}\right)f\right\|_{H^1(0,NT)}\leq Ce^{-\delta t}\|f\|_{H^1(0,NT)}
\end{equation}
for all $f\in H^1_{\rm per}(0,NT)$, where here $\mathcal{P}_{0,N}$ is the rank-one spectral projection onto the $NT$-periodic kernel of $\mathcal{A}[\phi]$:
see Remark~\ref{r:subharmonic}(ii) for more details. Equipped with this linear exponential decay result, the remainder of the proof of Theorem \ref{T:SS} follows from  standard 
nonlinear iteration arguments: for details, see \cite{SS19}. We point out that using Strichartz-type estimates, the nonlinear stability result in Theorem~\ref{T:SS} can be extended to $L^2$-subharmonic perturbations \cite{BD20}.

An important observation concerning Theorem \ref{T:SS} is that it lacks uniformity in $N$ in two (related) aspects.  Indeed, note that both the exponential  rate
of decay, specified by $\delta$, as well as the allowable size of initial perturbations, specified by $\eps=\eps_\delta$, are controlled completely in terms
of the size of the spectral gap $\delta_N>0$ of the linearized operator $\mathcal{A}[\phi]$ in \eqref{e:deltaN}.  Since $\delta_N\to 0$ as $N\to\infty$, it follows
that both $\delta$ and $\eps$ chosen in Theorem \ref{T:SS} necessarily tend to zero as $N\to\infty$ and that, consequently, the nonlinear stability
result Theorem \ref{T:SS} is empty in the limit $N=\infty$.  Note that at the linear level, the lack of uniformity in the allowable
size of the initial perturbations is due to the fact that $C=C_\delta\to\infty$ as $\delta\to 0$ in \eqref{e:SS_lin_est}.

In light of the above observations, it is therefore natural to ask if there is a way to obtain a stability result to subharmonic, i.e. $NT$-periodic, perturbations which is \emph{uniform in $N$}.  
In such a result, one should require that both the rate of decay and allowable size of the initial perturbations are uniform in $N$, thus depending only on the background
wave $\phi$.    At the linear level, this would correspond to proving the existence of a non-negative function $g:(0,\infty)\to(0,\infty)$ with $g(t)\to 0$
as $t\to\infty$ such that an inequality of the form
\[
\left\|e^{\mathcal{A}[\phi]t}\left(1-\mathcal{P}_{0,N}\right)\right\|_{\mathcal{L}(L^2_{\rm per}(0,NT))}
\leq g(t)
\]
holds for all $N\in\NM$ and $t>0$.  Our main result, stated below, shows that establishing such a uniform linear estimate is possible for the LLE \eqref{e:LLE}, with polynomial
rates of decay instead of exponential.  Further, it gives additional insight into the long-time dynamics of subharmonic perturbations.

\begin{theorem}[Uniform Subharmonic Linear Asymptotic Stability]\label{T:sub_main}
Suppose $\phi\in H^1_{\rm loc}(\RM)$ is a  $T$-periodic standing wave solution of \eqref{e:LLE} that is diffusively spectrally stable, in the sense
of Definition \ref{Def:spec_stab}. For each $N\in\NM$ let 
\[
\mathcal{P}_{0,N}:L^2_{\rm per}(0,NT)\to{\rm span}\{\phi'\}
\]
be the spectral projection of $L^2_{\rm per}(0,NT)$ onto the $NT$-periodic kernel of $\mathcal{A}[\phi]$.  
Then there exists a constant $C>0$ such that for every $N\in\NM$ and
$f\in L^1_{\rm per}(0,NT)\cap L^2_{\rm per}(0,NT)$ we have
\begin{equation}\label{result1}
\left\|e^{\mathcal{A}[\phi]t}(1-\mathcal{P}_{0,N})f\right\|_{L^2_{\rm per}(0,NT)}\leq C (1+t)^{-1/4}\|f\|_{L^1_{\rm per}(0,NT)\cap L^2_{\rm per}(0,NT)}
\end{equation}
valid for all $t>0$.  Furthermore,  there exists a constant $C>0$ such that
for all $N\in\NM$ and  $f\in L^1_{\rm per}(0,NT)\cap L^2_{\rm per}(0,NT)$ there exists a $NT$-periodic
function $\gamma_N(\cdot,t)=\gamma_N(\cdot,t;f)$  
\begin{equation}\label{result2}
\left\|\gamma_N(\cdot,t)-\frac{\left<\phi',\mathcal{P}_{0,N}f\right>_{L^2_{\rm per}(0,T)}}{\|\phi'\|^2_{L^2_{\rm per}(0,T)}}\right\|_{L^2_{\rm per}(0,NT)}
	\leq C(1+t)^{-1/4}\|f\|_{L^1_{\rm per}(0,NT)\cap L^2_{\rm per}(0,NT)}
\end{equation}
and
\begin{equation}\label{result3}
\left\|e^{\mathcal{A}[\phi]t}f-\phi'\gamma_N(\cdot,t)\right\|_{L^2_{\rm per}(0,NT)}
\leq C (1+t)^{-3/4}\|f\|_{L^1_{\rm per}(0,NT)\cap L^2_{\rm per}(0,NT)},
\end{equation}
for all $t>0$.
\end{theorem}

\begin{remark}
More than above, we show in Section \ref{S:sharp} below that the polynomial rates in Theorem \ref{T:sub_main} in fact provide \emph{sharp} uniform rates of decay
for subharmonic perturbations.   Moreover, observe that the uniform polynomial rates of decay require control of the initial data in $L^1$.  As we will see in our analysis,
this is due to the fact that linear diffusion equation does not exhibit decay from $L^2(\RM)$ to $L^2(\RM)$, but does from $L^1(\RM)$ to $L^2(\RM)$:  see Remark \ref{R:diffusive_decay} below.
\end{remark}

To interpret the above result, 
suppose that $\psi(x,t)$ is a solution of \eqref{e:LLE} with initial data $\psi(x,0)=\phi(x)+\eps f(x)$ with $|\eps|\ll 1$ and $f\in L^1_{\rm per}(0,NT)\cap L^2_{\rm per}(0,NT)$.
By \eqref{result1} it follows that, at the linear level, for $\eps\neq 0$ sufficiently small, the solution $\psi$ essentially behaves for large time like
\[
\psi(x,t)\approx\phi(x)+\eps\mathcal{P}_{0,N}f(x)=\phi(x)+\eps\frac{\left<\phi',\mathcal{P}_{0,N} f\right>_{L^2_{\rm per}(0,T)}}{\|\phi'\|_{L^2_{\rm per}(0,T)}^2}\phi'
	\approx\phi\left(x+\eps\frac{\left<\phi',\mathcal{P}_{0,N} f\right>_{L^2_{\rm per}(0,T)}}{\|\phi'\|_{L^2_{\rm per}(0,T)}^2}\right),
\]
corresponding to standard asymptotic (orbital) stability of $\phi$ with asymptotic phase.  This recovers (again, at the linear level) the asymptotic
stability result in Theorem \ref{T:SS}, but now with asymptotic rates of decay which are uniform in $N$.  Further than this,
Theorem \ref{T:sub_main} implies that there exists a function $\gamma_N(x,t)$ which is $NT$-periodic in $x$ such that
\[
\psi(x,t)\approx\phi(x)+\eps\gamma_N(x,t)\phi'(x)\approx\phi(x+\eps\gamma_N(x,t)),~~t\gg 1,
\]
with, by \eqref{result3}, a faster rate of convergence (in time), giving a refined insight into the long-time local dynamics near $\phi$ (described by a space-time dependent translational modulation) beyond
the more standard asymptotic stability as in Theorem \ref{T:SS}.  As a consistency check, note that \eqref{result2} implies that as $t\to\infty$ the function $\gamma_N(\cdot,t)$ tends to the asymptotic
phase predicted by \eqref{result1}.  As we will see, the incorporation of such a space-time dependent modulation function is key to our analysis
and is precisely what allows us to obtain such uniform decay rates.   

The key observation is that the bounds on the evolution semigroups and the linear rates of decay obtained in Theorem \ref{T:sub_main} on subharmonic perturbations are uniform in $N$.  It turns out that these uniform 
decay rates are precisely the linear decay rates one obtains by considering the semigroup $e^{\mathcal{A}[\phi]t}$ as acting on $L^2(\RM)$, i.e. they
agree exactly with the linear rates of decay for localized perturbations.  That the decay rate to localized perturbations should uniformly control
all subharmonic perturbations may be formally motivated by observing that, up to appropriate translations, a sequence of $NT$-periodic functions may converge 
as $N\to\infty$ to a function in $L^2(\RM)$ locally in space.

Based on the above comments, it shouldn't be surprising that the general methodology used for the proof of Theorem \ref{T:sub_main} is modeled off of the associated linear analysis
to localized perturbations.  The localized analysis, in turn, is largely based off the work \cite{JNRZ_Invent} and, for completeness and to motivate
the approach towards the proof of Theorem \ref{T:sub_main}, we review the localized analysis in Section \ref{S:stab_loc} below.  
In particular, we obtain the following result.

\begin{theorem}[Localized Linear Asymptotic Stability]\label{t:stab}
Suppose $\phi\in H^1_{\rm loc}(\RM)$ is a  $T$-periodic standing wave solution of \eqref{e:LLE} that is diffusively spectrally stable, in the sense
of Definition \ref{Def:spec_stab}.
Then there exists a constant $C>0$ such that for any $f\in L^1(\RM)\cap L^2(\RM)$  we have 
\[
\left\|e^{A[\phi]t}f\right\|_{L^2(\RM)}\leq C(1+t)^{-1/4}\|f\|_{L^1(\RM)\cap L^2(\RM)},
\]
for all $t>0$. Furthermore,  there exists a constant $C>0$ such that for each $f\in L^1(\RM)\cap L^2(\RM)$ there exists a function $\gamma(\cdot,t)=\gamma(\cdot,t;f)$ such that
\[
\|\gamma(\cdot,t)\|_{L^2(\RM)}\leq C(1+t)^{-1/4}\|f\|_{L^1(\RM)\cap L^2(\RM)}
\]
and
\[
\left\|e^{A[\phi]t}f-\phi'\gamma(\cdot,t)\right\|_{L^2(\RM)}\leq C(1+t)^{-3/4}\|f\|_{L^1(\RM)\cap L^2(\RM)},
\]
for all $t>0$.
\end{theorem}

The proofs of both Theorem \ref{T:sub_main} and Theorem \ref{t:stab}  rely on a delicate decomposition of the semigroup
$e^{\mathcal{A}[\phi]t}$ acting on the appropriate underlying space ($L^2_{\rm per}(0,NT)$ and $L^2(\RM)$, respectively).  While the decomposition
is similar in the two cases, the subharmonic result in Theorem \ref{T:sub_main} requires an additional level of decomposition which is not needed in the localized case. As we will see in Section \ref{S:stab_per} below, our proof of Theorem \ref{T:sub_main} connects to both the exponential decay result in Theorem \ref{T:SS}
as well as the localized result in Theorem \ref{t:stab}.  Indeed, we will see that if one fixes $N\in\NM$ then the (linear) exponential decay of $NT$-periodic
perturbations follows naturally from our methodology.  Further, by formally taking $N\to\infty$ we see that the result in Theorem \ref{T:sub_main}
recovers the localized result in Theorem \ref{t:stab}.  For instance, we will see that, formally at least, the $NT$-periodic modulation function $\gamma_N$ in Theorem
\ref{T:sub_main} satisfies
\[
\lim_{N\to\infty}\gamma_N(x,t)=\gamma(x,t),
\]
where here $\gamma$ is the localized modulation function from Theorem \ref{t:stab}.  In this way, our work both expands and unifies several previous works in the literature.

\begin{remark}
During our proof of Theorem \ref{T:sub_main}, we will also see how the techniques presented provide exponential decay results of the form \eqref{e:SS_lin_est} with a constant
$C>0$ which is \emph{uniform} in $N$.  This extends the key linear estimate \eqref{e:SS_lin_est} used in \cite{SS19} to establish Theorem \ref{T:SS}.
\end{remark}

We also emphasize that our methodology used in the proofs of Theorem \ref{T:sub_main} and Theorem \ref{t:stab} is quite
general and applies more broadly than for the LLE \eqref{e:LLE}.  
As mentioned previously, our arguments are motivated by the recent work \cite{JNRZ_Invent} which, in turn, was based on a sequence
of previous works \cite{JZ10,JZN,JZ_11_1,JZ_11_2,JNRZ_13_1,JNRZ_13_2,BJNRZ_13}  by the same authors, all of which were eventually based on the seminal work of Schneider \cite{S98_1,S98_2}.  
In fact, our work relies on only a few key features of the linearized operator $\mathcal{A}[\phi]$.   Namely, Theorems \ref{T:sub_main} and Theorem \ref{t:stab}
continue to hold provided the following properties are satisfied:
\begin{enumerate}
\item The wave $\phi$ is diffusively spectrally stable, as defined in Definition \ref{Def:spec_stab}.
\item The operator $\mathcal{A}[\phi]$ generates  $C^0$-semigroups on $L^2(\RM)$ and $L^2_{\rm per}(0,NT)$, and for each $\xi\in[-\pi/T,\pi/T)$ the Bloch operators $\mathcal{A}_\xi[\phi]$ generate $C^0$-semigroups on $L^2_{\rm per}(0,T)$.
\item There exist positive constants $\mu_0$ and $C_0$ such that for each $\xi\in[-\pi/T,\pi/T)$ the Bloch resolvent operators satisfy
    \begin{equation}\label{e:resestg}
\|(i\mu-\mathcal{A}_\xi[\phi])^{-1}\|_{\mathcal L(L ^2_{\rm per}(0,T))}\leq C_0,~~{\rm for~all}~|\mu|>\mu_0.
\end{equation}
\end{enumerate}
Consequently, our work sets forth a general methodology for establishing analogous (linear) results to Theorem \ref{T:sub_main} and Theorem \ref{t:stab}
in more general contexts.  Further, note that the conditions (i)-(iii) above are slightly more general than the corresponding assumptions used in \cite{JNRZ_Invent}.
For our analysis of the LLE \eqref{e:LLE}, the first property above is the main assumption, that we know it holds at least for the periodic waves $\phi_\mu$ given by \eqref{e:periodic}, and the other two properties are proved in Section~\ref{ss:existsg}. 

Of course, it is natural to ask if our linear results can be extended to a result pertaining to the nonlinear dynamics of the LLE \eqref{e:LLE}.  
Using a nonlinear iteration scheme to establish such a nonlinear result, in Section \ref{S:nonlinear} below we will see that the structure of \eqref{e:LLE} implies such an iteration induces a loss of derivatives when considering space-time dependent modulations.  Such a phenomena is well known in the context
of reaction diffusion equations and systems of conservation laws, where the loss of derivatives can be compensated by a nonlinear damping effect which slaves
high Sobolev norms to low Sobolev norms.  However, such nonlinear damping techniques rely heavily on the damping in the governing evolution equation to correspond
to the highest order spatial derivative present.  In the case of the LLE, unfortunately, damping appears as the lowest-order derivative and hence one has
no hope of regaining derivatives through such nonlinear damping estimates.  Consequently, obtaining corresponding nonlinear stability results for the LLE \eqref{e:LLE}
is still an open problem\footnote{The ability to establish nonlinear results when one has the ability to regain these lost derivatives is currently under investigation by the authors.}.

\medskip

The outline for this paper is as follows.  
In Section \ref{S:prelim}, we review several preliminary results, including a review of Floquet-Bloch theory in the context of both
localized (Section \ref{S:bloch_loc}) and $NT$-periodic (Section \ref{S:bloch_per}) functions.  Specifically, we provide a characterization
of the spectrum of $\mathcal{A}[\phi]$ on $L^2(\RM)$ and  $L^2_{\rm per}(0,NT)$ in terms of the associated Bloch operators. In Section~\ref{ss:existsg}, we collect the properties of  these Bloch operators required for our analysis. We describe their spectral properties, then establish the existence and basic decay properties of the corresponding Bloch semigroups.
Section \ref{S:stab_loc} is dedicated to the proof of the localized result Theorem \ref{t:stab}, which in turn serves as motivation
for our proof of Theorem \ref{T:sub_main}, which is presented in Section \ref{S:stab_per}.  
In Section \ref{S:sharp}, we present a technical bound which establishes that the decay rates for localized perturbations in Theorem \ref{t:stab} in fact provide
\emph{sharp} uniform decay rates for subharmonic perturbations.  The proof of this key bound is given in the Appendix.
Finally, in Section \ref{S:nonlinear} we describe
the mathematical challenges encountered in establishing the corresponding nonlinear stability of diffusively spectrally stable periodic standing solutions of the LLE \eqref{e:LLE}.
Throughout our work, we aim to
make clear the ways in which our analysis unifies and expands previous works, as well as its ability to be generalized to other contexts.

\medskip

\noindent
{\bf Acknowledgments:} The work of MAJ and WRP was partially supported by the NSF under grant DMS-1614785.  MAJ was additionally supported
by the Simons Foundation Collaboration Grant number 714021.  MH was partially supported by the EUR EIPHI program (Contract No. ANR-17-EURE-0002) and the ISITE-BFC project (Contract No. ANR-15-IDEX-0003).  The authors also thank the referee for their helpful comments.

\section{Preliminaries}\label{S:prelim}


In the first two subsections below, we review some general results from the Floquet-Bloch theory for both  localized and subharmonic perturbations.  We then give some spectral properties and semigroup estimates for the Bloch operators obtained from the LLE \eqref{e:LLE}.

\subsection{Floquet-Bloch Theory for Localized Perturbations}\label{S:bloch_loc}

Consider a  matrix differential operator $\mathcal{A}$  with $T$-periodic coefficients which belong to $H^1_{\rm loc}(\RM)$. 
Assume that $\mathcal{A}$ is closed when acting on\footnote{Recall that we are writing $L^2(\RM)$ throughout rather than the technically more correct $(L^2(\RM))^n$ for an $n\times n$ matrix operator.}
$L^2(\RM)$ with domain $H^s(\RM)$, for some $s\geq1$.
Floquet theory implies for each $\lambda\in\CM$ that non-trivial solutions of the ordinary differential equation
\[
\mathcal{A}v=\lambda v
\]
cannot be integrable on $\RM$ and that, at best, they can be bounded functions on the real line (e.g., see \cite{KP_book,RS4}).  In particular,
the $L^2(\RM)$ spectrum of $\mathcal{A}$ can contain no eigenvalues, and hence must be entirely essential.

To characterize the essential spectrum of $\mathcal{A}$, note again by Floquet theory that any bounded solution of the above spectral problem must be of the form
\[
v(x)=e^{i\xi x}w(x),
\]
for some $T$-periodic function $w$ and constant $\xi\in[-\pi/T,\pi/T)$.  From these observations, it can be shown that $\lambda\in \CM$ belongs to the 
$L^2(\RM)$-spectrum of $\mathcal{A}$ if and only if  the problem
\begin{equation}\label{e:quasiper}
\begin{cases}
\mathcal{A}v=\lambda v\\
v(x+T)=e^{i\xi T}v(x)
\end{cases}
\end{equation}
admits a non-trivial solution for some $\xi\in[-\pi/T,\pi/T)$.  Equivalently, setting $v(x)=e^{i\xi x}w(x)$ in \eqref{e:quasiper} for some $T$-periodic $w$, 
we see that \eqref{e:quasiper} holds if and only if there exists
a $\xi\in[-\pi/T,\pi/T)$ and a non-trivial $w\in L^2_{\rm per}(0,T)$ such that
\[
\lambda w=e^{-i\xi x}\mathcal{A}e^{i\xi x}w=:\mathcal{A}_\xi w.
\]
The operators $\mathcal{A}_\xi$
are referred to as the Bloch operators associated to $\mathcal{A}$ and $\xi$ is referred to as the Bloch frequency. 
Each $\mathcal{A}_\xi$ acts on the space of $T$-periodic functions $L^2_{\rm per}(0,T)$, on which they are closed with dense and compactly embedded domain $H^s_{\rm per}(0,T)$.

Note that since the domains of the Bloch operators are compactly embedded in 
$L^2_{\rm per}(0,T)$, their spectra consist entirely of isolated eigenvalues of finite algebraic multiplicities which, furthermore, depend continuously
on the Bloch parameter $\xi$.  By the above Floquet-Bloch theory, we in fact have the spectral decomposition
\begin{equation}\label{e:bloch_spec}
  \sigma_{L^2(\RM)}\left(\mathcal{A}\right)=\bigcup_{\xi\in[-\pi/T,\pi/T)}\sigma_{L^2_{\rm per}(0,T)}\left(\mathcal{A}_\xi\right);
\end{equation}
see, for example, \cite{G1}.
This characterizes the $L^2(\RM)$ spectrum of $\mathcal{A}$ as the union of countably many continuous curves $\lambda(\xi)$
corresponding to the eigenvalues of the associated Bloch operators $\mathcal{A}_\xi$.

From the above characterization of the spectrum of $\mathcal{A}$, it is clearly desirable to have the ability to decompose arbitrary 
functions in $L^2(\RM)$ into superpositions of functions of the form $e^{i\xi x}w(x)$ with $\xi\in[-\pi/T,\pi/T)$ and $w\in L^2_{\rm per}(0,T)$.
This may be accomplished by noting that any function $g\in L^2(\RM)$ admits a Bloch decomposition, or inverse Bloch transform representation,
given by
\begin{equation}\label{e:Bloch}
g(x)=\frac{1}{2\pi}\int_{-\pi/T}^{\pi/T}e^{i\xi x}\check{g}(\xi,x)d\xi,~~{\rm where}~\check{g}(\xi,x):=\sum_{\ell\in\ZM}e^{2\pi i\ell x/T}\hat{g}(\xi+2\pi \ell/T)
\end{equation}
and $\hat{g}(\cdot)$ denotes the Fourier transform of $g$, defined here as $\hat{g}(\xi)= \int_{-\infty}^\infty e^{-i\xi x}g(x)dx$, $\xi\in\R$.
Indeed, observe that for any Schwartz function $g$ we have 
that
\[
 2\pi g(x)=\int_{-\infty}^\infty e^{i\xi x}\hat{g}(\xi)d\xi=\sum_{\ell\in\ZM}\int_{-\pi/T}^{\pi/T} e^{i(\xi+2\pi \ell/T)x}\hat{g}(\xi+2\pi \ell/T)d\xi=\int_{-\pi/T}^{\pi/T}e^{i\xi x}\check{g}(\xi,x)d\xi,
\]
and then the general result follows by density.
Note that for each fixed $\xi\in[-\pi/T,\pi/T)$ the function $\check{g}(\xi,\cdot)$ is  $T$-periodic  and hence the above procedure decomposes, as desired,
arbitrary functions in $L^2(\RM)$ into a (continuous) superposition of functions of the form  $e^{i\xi \cdot}\check{g}(\xi,\cdot)$, each of which has a fixed Bloch frequency $\xi$.

Defining the Bloch transform
\[
\mathcal{B}:L^2(\RM)\to L^2\left([-\pi/T,\pi/T);L^2_{\rm per}(0,T)\right)
\]
via the action $g\mapsto\check{g}$ with $\check{g}$ given by \eqref{e:Bloch},
the standard Parseval identity for the Fourier transform implies that the Bloch transform is a bounded linear operator,
\begin{equation}\label{e:parseval_loc}
  \|g\|^2_{L^2(\RM)}=\frac{1}{2\pi T}
  \int_{-\pi/T}^{\pi/T}\int_0^T\left|\mathcal{B}(g)(\xi,x)\right|^2dx~d\xi
=\frac{1}{2\pi T} \|\mathcal B(g)\|^2_{L^2\left([-\pi/T,\pi/T);L^2_{\rm per}(0,T)\right)}.
\end{equation}
Furthermore,  for $v\in H^s(\RM)$ we have
\[
\mathcal{B}\left(\mathcal{A}v\right)(\xi,x)=\left(\mathcal{A}_\xi\check{v}(\xi,\cdot)\right)(x)\quad\mbox{and}\quad
\mathcal{A}v(x)=\frac{1}{2\pi}\int_{-\pi/T}^{\pi/T} e^{i\xi x}\mathcal{A}_\xi\check{v}(\xi,x)d\xi,
\] 
showing that the Bloch transform $\mathcal{B}$ diagonalizes the $T$-periodic coefficient linear operator $\mathcal{A}$ in the same
way that the standard Fourier transform diagonalizes constant coefficient linear operators.
Here, the Bloch operators $\mathcal{A}_\xi$ may be viewed as operator-valued symbols of the linearized operator $\mathcal{A}$  under the action of $\mathcal{B}$. Furthermore, assuming that $\mathcal{A}$ and its associated Bloch operators
$\mathcal{A}_\xi$ generate $C^0$-semigroups on $L^2(\RM)$ and $L^2_{\rm per}(0,T)$, respectively, it is straightforward to check the identities 
\begin{equation}\label{e:blochsoln}
\mathcal{B}\left(e^{\mathcal{A}t}v\right)(\xi,x)=\left(e^{\mathcal{A}_\xi t}\check{v}(\xi,\cdot)\right)(x)
\quad\mbox{and}\quad
e^{\mathcal{A}t}v(x)=\frac{1}{2\pi}\int_{-\pi/T}^{\pi/T} e^{i\xi x}e^{\mathcal{A}_\xi t}\check{v}(\xi,x)d\xi.
\end{equation}

\subsection{Floquet-Bloch Theory for Subharmonic Perturbations}\label{S:bloch_per}

In this section we introduce a version of Floquet-Bloch theory that is appropriate for the study of subharmonic perturbations.  
Specifically, given a differential
operator $\mathcal{A}$ with $T$-periodic coefficients, 
as in the previous section,  we aim at understanding how the associated semigroup $e^{\mathcal{A}t}$ acts
on $NT$-periodic functions.  
To this end, first observe from \eqref{e:quasiper} that if we define the set
\begin{equation}\label{e:freq}
\Omega_N:=\left\{\xi\in[-\pi/T,\pi/T):e^{i\xi NT}=1\right\},
\end{equation}
then the perturbation $v$ satisfies $NT$-periodic boundary conditions precisely when $\xi\in\Omega_N$.  In particular, by \eqref{e:quasiper} we have the  equality
\[
\sigma_{L_{\rm per}^2(0,NT)}\left(\mathcal{A}\right)=\bigcup_{\xi\in\Omega_N}\sigma_{L^2_{\rm per}(0,T)}\left(\mathcal{A}_\xi\right),
\]
providing a description in terms of Bloch operators of the $NT$-periodic spectrum of $\mathcal{A}$.  
Note that the set $\Omega_N$ may be written explicitly when $N$ is even by
\[
\Omega_N=\left\{\xi_j=\frac{2\pi j}{NT}:j=-\frac{N}{2},~-\frac{N}{2}+1,\ldots,\frac{N}{2}-1\right\}
\]
and when $N$ is odd by
\[
\Omega_N=\left\{\xi_j=\frac{2\pi j}{NT}:j=-\frac{N-1}{2},-\frac{N-1}{2}+1,\dots,\frac{N-1}{2}\right\}.
\]
In particular, for each $N\in\NM$ we have $0\in\Omega_N$, $|\Omega_N|=N$ and $\Delta\xi_j:=\xi_j-\xi_{j-1}=2\pi/NT$ for each~$j$.

We now define, for each $\xi\in\Omega_N$, the $T$-periodic Bloch transform of a function $g\in L^2_{\rm per}(0,NT)$~as
\begin{equation}\label{e:Bloch_Tper}
\mathcal{B}_T(g)(\xi,x):=\sum_{\ell\in\ZM} e^{2\pi i\ell x/T}\hat{g}(\xi+2\pi\ell/T),
\end{equation}
where now $\hat{g}$ represents the Fourier transform of $g$ on the torus given by
\[
\hat{g}(\zeta)=\int_{-NT/2}^{NT/2}e^{-i\zeta y}g(y)dy.
\]
Note that $\hat{g}(\zeta)$ is the $k$-th Fourier coefficient of the $NT$-periodic function $g$ when  $\zeta=2\pi k/NT$, and that  $\mathcal{B}_T(g)(\xi,\cdot)$ is clearly a $T$-periodic function for all $\xi\in\Omega_N$. We recover $g$ through the Fourier series representation
\[
g(x)=\frac{1}{NT}\sum_{k\in\ZM}e^{2\pi ikx/NT}\hat{g}(2\pi k/NT),
\]
which together with the readily checked identity
\[
\sum_{k\in\ZM}f(2\pi k/NT)=\sum_{\xi\in\Omega_N}\sum_{\ell\in\Z} f(\xi + 2\pi \ell/T),
\]
valid for any $f$ for which the sum converges, gives the identity
\[
g(x)= \frac{1}{NT}\sum_{\xi\in\Omega_N}\sum_{\ell\in\Z} e^{i(\xi + 2\pi\ell/T)x}\hat{g}(\xi+2\pi \ell /T).
\]
This yields the inverse Bloch representation formula
\begin{equation}\label{e:inv_b_rep}
g(x)=\frac{1}{NT}\sum_{\xi\in\Omega_N}e^{i\xi x}\mathcal{B}_T(g)(\xi,x)
\end{equation}
which is valid for all $g\in L^2_{\rm per}(0,NT)$.
The equalities \eqref{e:Bloch_Tper} and \eqref{e:inv_b_rep} are the analogue of \eqref{e:Bloch} for $NT$-periodic functions.

Further, notice the following subharmonic Parseval identity
\begin{equation}\label{e:parseval_per}
\left<f,g\right>_{L^2(0,NT)} = \frac{1}{NT^2}\sum_{\xi\in\Omega_N}\left<\mathcal{B}_T(f)(\xi,\cdot),\mathcal{B}_T(g)(\xi,\cdot)\right>_{L^2(0,T)},
\end{equation}
valid for all $f,g\in L^2_{\rm per}(0,NT)$.  Indeed, observe that
\begin{align*}
\sum_{\xi\in\Omega_N}&\left<\mathcal{B}_T(f)(\xi,\cdot),\mathcal{B}_T(g)(\xi,\cdot)\right>_{L^2(0,T)} = 
	\sum_{\xi\in\Omega_N}\int_0^T \overline{\mathcal{B}_T(f)(\xi,x)}\mathcal{B}_T(g)(\xi,x) dx\\
&\qquad\qquad= \sum_{\xi\in\Omega_N}\int_0^T \sum_{k,\ell\in\Z}e^{2\pi i(\ell-k)x/T} \overline{\hat{f}(\xi + 2\pi k/T)}\hat{g}(\xi+2\pi\ell/T) dx\\
&\qquad\qquad= \sum_{\xi\in\Omega_N} \sum_{k,\ell\in\Z}\left(\int_0^Te^{2\pi i(\ell-k)x/T}dx\right) \overline{\hat{f}(\xi + 2\pi k/T)} \hat{g}(\xi+2\pi\ell/T)\\
&\qquad\qquad= T\sum_{\xi\in\Omega_N}\sum_{\ell\in\Z}\overline{\hat{f}(\xi + 2\pi \ell/T)} \hat{g}(\xi+2\pi\ell/T)\\
&\qquad\qquad= T\sum_{k\in\Z}\overline{\hat{f}\left(\frac{2\pi k}{NT}\right)}\hat{g}\left(\frac{2\pi k}{NT}\right) = NT^2 \left<f,g\right>_{L^2(0,NT)},
\end{align*}
as claimed. In particular, \eqref{e:parseval_per} implies that $\mathcal B_T$ is a bounded linear operator, just as the Bloch transform $\mathcal B$ in the case of localized perturbations. It also leads to the following lemma needed later in our analysis.

\begin{lemma}\label{L:per_bloch}
Let $N\in\NM$.  If $f\in L^2_{\rm per}(0,T)$ and $g\in L^2_{\rm per}(0,NT)$, then
\[
\mathcal{B}_T(fg)(\xi,x)=f(x)\mathcal{B}_T(g)(\xi,x),
\]
In particular, for such $f$ and $g$  we have the identity
\[
\left<f,g\right>_{L^2(0,NT)}=\frac{1}{T}\left<f(x),\mathcal{B}_T(g)(0,x)\right>_{L^2(0,T)}.
\]
\end{lemma}
\begin{proof}
For $f\in L^2_{\rm per}(0,T)$ and $g\in L^2_{\rm per}(0,NT)$,  we calculate the $NT$-periodic Fourier transform as
\begin{align*}
\widehat{fg}(z)&=\int_0^{NT}e^{-izy}f(y)g(y)dy\\
&=\int_0^{NT}e^{-izy}\left(\frac{1}{T}\sum_{k\in\ZM}e^{2\pi iky/T}\hat{f}\left({2\pi k}/{T}\right)\right)g(y)dy\\
&=\frac{1}{T}\sum_{k\in\ZM}\hat{f}\left({2\pi k}/{T}\right)\left(\int_0^{NT}e^{-i(z-2\pi k/T)y}g(y)dy\right)\\
&=\frac{1}{T}\sum_{k\in\ZM}\hat{f}\left({2\pi k}/{T}\right)\hat{g}(z-{2\pi k}/{T}).
\end{align*}
Thus, we have
\begin{align*}
\mathcal{B}_T(fg)(\xi,x)&=\frac{1}{T}\sum_{\ell\in\ZM}e^{2\pi i\ell x/T}\left(\sum_{k\in\ZM}\hat{f}({2\pi k}/{T})\hat{g}(\xi+{2\pi (\ell-k)}/{T})\right)\\
&=\frac{1}{T}\sum_{k\in\ZM}\hat{f}({2\pi k}/{T})\sum_{\ell\in\ZM} e^{2\pi i\ell x/T}\hat{g}\left(\xi+{2\pi (\ell-k)}/{T}\right)\\
&=\frac{1}{T}\sum_{k\in\ZM}e^{2\pi ikx/T}\hat{f}\left({2\pi k}/{T}\right)\sum_{\ell\in\ZM} e^{2\pi i\ell x/T}\hat{g}\left(\xi+{2\pi \ell}/{T}\right)\\
&=f(x)\mathcal{B}_T(g)(\xi,x),
\end{align*}
which proves the first equality.

Next, from this equality and Parseval's identity \eqref{e:parseval_per} we obtain that
\begin{align*}
\left<f,g\right>_{L^2(0,NT)}=\frac{1}{NT^2}\sum_{\xi\in\Omega_N}\left<f(x)\mathcal{B}_T(1)(\xi,\cdot),\mathcal{B}_T(g)(\xi,\cdot)\right>_{L^2(0,T)}.
\end{align*}
Noting that
\[
\hat{1}\left(\xi+2\pi \ell/T\right)=\begin{cases}
NT, &{\rm if }~~\xi+2\pi \ell/T=0\\
0, &{\rm otherwise}
\end{cases}
\]
and that the condition $\xi+2\pi \ell/T=0$ holds for $\xi\in\Omega_N$ and $\ell\in\ZM$ if and only if $\xi=\ell=0$, we find that
$\mathcal B_T(1)(\xi,x)=NT$, which proves the second equality.
\end{proof}

Similarly to the case of $L^2(\R)$, we have the connections between the operator $\mathcal{A}$ and its associated Bloch operators $\mathcal{A}_\xi$,
\[
\mathcal{A}g(x)=\frac{1}{NT}\sum_{\xi\in\Omega_N}e^{i\xi x}\mathcal{A}_\xi \mathcal{B}_T(g)(\xi,x),\quad g\in H^s_{\rm per}(0,NT),
\]
and then assuming they generate $C^0$-semigroups, also between the corresponding semigroups,
\begin{equation}\label{e:per_semigrp}
e^{\mathcal{A}t}g(x)=\frac{1}{NT}\sum_{\xi\in\Omega_N}e^{i\xi x}e^{\mathcal{A}_\xi t}\mathcal{B}_T(g)(\xi,x),\quad g\in L^2_{\rm per}(0,NT).
\end{equation}

An important observation here is that since \eqref{e:inv_b_rep} can be rewritten as
\[
g(x)=\frac{1}{2\pi}\sum_{\xi\in\Omega_N}e^{i\xi x}\mathcal{B}_T(g)(\xi,x)\Delta\xi,
\quad \Delta\xi=\frac{2\pi}{NT},
\]
the representation \eqref{e:inv_b_rep} has the form of a Riemann sum approximation of the Bloch decomposition formula  \eqref{e:Bloch} for functions $g\in L^2(\R)$.  Similarly, the representation in \eqref{e:per_semigrp} may be considered as a Riemann sum approximation of 
the formula \eqref{e:blochsoln}.  These interpretations of the above identities will be crucial in the forthcoming analysis.

\begin{remark}
In the case $N=1$,  corresponding to co-periodic perturbations, we have $\Omega_1=\{0\}$ and hence
the Bloch transform $\mathcal{B}_T(g)$ simply recovers the Fourier series representation for $g\in L^2_{\rm per}(0,T)$.  Furthermore, the representation formulas for $\mathcal A$ and $e^{\mathcal{A}t}$ given above are reduced to the obvious equalities $\mathcal A=\mathcal A_0$ and $e^{\mathcal{A}t}=e^{\mathcal{A}_0t}$.
\end{remark}

\subsection{Properties of Bloch Semigroups }\label{ss:existsg}

Now we restrict to the linear operators $\mathcal A[\phi]$ and  $\mathcal A_\xi[\phi]$ found from the LLE \eqref{e:LLE}, assuming that $\phi\in H^1_{\rm loc}(\RM)$ is a diffusively spectrally stable $T$-periodic stationary solution of the profile equation \eqref{e:profile}. Recall that $\mathcal A[\phi]$ is given by \eqref{e:Aphi} from which we find the Bloch operators  
\begin{equation}\label{e:Aphixi}
\mathcal A_\xi[\phi]=- I+\mathcal{J}\mathcal{L}_\xi[\phi],
\end{equation}
with
\[
\mathcal{J}=\left(\begin{array}{cc}0&-1\\1&0\end{array}\right),\quad 
\mathcal{L}_\xi[\phi] = \left(\begin{array}{cc} -\b (\partial_x+i\xi)^2 - \a  + 3\phi_{r}^2 + \phi_{i}^2 & 2\phi_{r}\phi_{i} \\
  2\phi_{r}\phi_{i} & -\b (\partial_x+i\xi)^2 - \a  + \phi_{r}^2 + 3\phi_{i}^2\end{array}\right).
\]

Our first lemma summarizes the spectral properties of the Bloch operators $\mathcal A_\xi[\phi]$ which directly follow from the definition of diffusive spectral stability.

\begin{lemma}[Spectral Preparation]\label{L:spec_prep}
  Suppose $\phi\in H^1_{\rm loc}(\RM)$ is a diffusively spectrally stable $T$-periodic stationary solution of the LLE \eqref{e:LLE}.  Then the following properties hold.
  \begin{enumerate}
\item For any fixed $\xi_0\in(0,\pi/T)$, there exists a positive constant $\delta_0$  such that
\[
\Re\sigma(\mathcal{A}_\xi[\phi])<-\delta_0,
\]
for all $\xi\in[-\pi/T,\pi/T)$ with $|\xi|>\xi_0$.
  \item There exist positive constants $\xi_1$, $\delta_1$, and $d$ such that for any $|\xi|<\xi_1$ the spectrum of $\mathcal{A}_\xi[\phi]$ decomposes into two
disjoint subsets
\[
\sigma(\mathcal{A}_\xi[\phi])=\sigma_-(\mathcal{A}_\xi[\phi])\cup\sigma_0(\mathcal{A}_\xi[\phi]),
\]
with the following properties:
\begin{enumerate}
\item $\Re\sigma_-(\mathcal{A}_\xi[\phi])<-\delta_1$ and $\Re\sigma_0(\mathcal{A}_\xi[\phi])>-\delta_1$;
\item the set $\sigma_0(\mathcal{A}_\xi[\phi])$ consists of a single negative eigenvalue $\lambda_c(\xi)$ which is analytic in $\xi$ and expands as
\[
\lambda_c(\xi)=ia\xi-d\xi^2+\mathcal{O}(\xi^3),
\]
for $|\xi|\ll 1$ for some $a\in\RM$ and $d>0$;
\item the eigenfunction associated to $\lambda_c(\xi)$ is analytic near $\xi=0$ and expands as
\[
\Phi_\xi(x)=\phi'(x)+\mathcal{O}(\xi),
\]
where $\phi'$ is the derivative of the $T$-periodic solution $\phi$.
\end{enumerate}
\end{enumerate}
\end{lemma}

\begin{proof} The first part is an immediate consequence of the properties (i) and (ii) in the Definition~\ref{Def:spec_stab}. To prove the second part, observe that since $\lambda=0$ is a simple, isolated eigenvalue of $\mathcal{A}_0[\phi]$, and since $\mathcal{A}_\xi[\phi]$ depend continuously on $\xi$, standard spectral perturbation theory implies the continuous dependence on $\xi$ of the eigenvalue $\lambda_c(\xi)$ and of its associated eigenvector $\Phi_\xi(x)$. 
\end{proof}

\begin{remark}
In the case where the background wave $\phi$ is even (up to translation), a simple symmetry argument implies that $a=0$ in the expansion of $\lambda_c(\xi)$ above.  
Note that the diffusively spectrally stable solutions \eqref{e:periodic} constructed and studied via bifurcation theory in \cite{DH18_2,DH18_1}
are all even.  For more general waves, however, one may have $a\neq 0$. 
\end{remark}

Our next result establishes that the linearized operator $\mathcal{A}[\phi]$ and its associated Bloch operators generate $C^0$-semigroups.  The proof
is elementary and relies on a decomposition of these operators into a constant coefficient operator plus a bounded perturbation.  
Specifically, we establish the following result (see also \cite[Lemma 1]{SS19}).

\begin{lemma}\label{L:C0}
Assume that $\phi\in H^1_{\rm loc}(\RM)$ is  a $T$-periodic solution  of the stationary LLE \eqref{e:profile}. The linear operator  $\mathcal A[\phi]$ acting in either $L^2(\R)$ 
or $L^2_{\rm per}(0,NT)$ generates a $C^0$ semigroup.  Similarly, for each $\xi\in[-\pi/T,\pi/T)$ the Bloch operators $\mathcal A_\xi[\phi]$ acting in $L^2_{\rm per}(0,T)$  generate $C^0$-semigroups.
\end{lemma}

\begin{proof} The proofs being the same in the three cases, we only consider the  Bloch operators $\mathcal A_\xi[\phi]$ acting in $L^2_{\rm per}(0,T)$. 
From \eqref{e:Aphixi} we see that $\mathcal A_\xi[\phi]$ is a bounded perturbation of the operator \[\mathcal A_\xi^0=-\beta(\partial_x+i\xi)^2\mathcal J.\] Since bounded perturbations of generators of  $C^0$-semigroups also generate a $C^0$-semigroups \cite[Theorem 1.1, Section 3.1]{Pazy},
it remains to prove the result for the operator $\mathcal A_\xi^0$. This operator is closed on  $L^2_{\rm per}(0,T)$ with domain  $H^2_{\rm per}(0,T)$ and has constant coefficients. Using Fourier analysis, it is then straightforward to check that its spectrum lies on the imaginary axis $i\R$ and that for any complex number $\lambda$ in its resolvent set the norm of the resolvent operator is given by
  \[
\|(\lambda-\mathcal{A}_\xi^0)^{-1}\|_{\mathcal L(L ^2_{\rm per}(0,T))} = \frac1{\mbox{dist}(\lambda, \sigma(\mathcal{A}_\xi^0))}.
\]
Consequently, for any complex number $\lambda$ with $\Re\lambda>0$ we have 
\[
\|(\lambda-\mathcal{A}_\xi^0)^{-1}\|_{\mathcal L(L ^2_{\rm per}(0,T))} = \frac1{\Re\lambda}.
\]
Together with the Hille-Yosida theorem (e.g., see \cite[Chapter IX.2]{Kato}) this implies that $\mathcal{A}_\xi^0$ generates a $C^0$-semigroup which proves the lemma.
\end{proof}

The next step consists in proving the resolvent estimate \eqref{e:resestg} for the Bloch operators $\mathcal{A}_\xi[\phi]$ given by the LLE \eqref{e:LLE}\footnote{For the periodic waves $\phi_\mu$ given by \eqref{e:periodic} this result can be easily proved using perturbation arguments, since the operators are in this case small bounded perturbations of operators with constant coefficients.}.

\begin{lemma}\label{l:resolvent}
  Suppose $\phi\in H^1_{\rm loc}(\RM)$ is a diffusively spectrally stable $T$-periodic stationary solution of the LLE \eqref{e:LLE}. 
There exist positive constants $\mu_0$ and $C_0$ such that for each $\xi\in[-\pi/T,\pi/T)$ the Bloch resolvent operators satisfy
    \begin{equation}\label{e:resest}
\|(i\mu-\mathcal{A}_\xi[\phi])^{-1}\|_{\mathcal L(L ^2_{\rm per}(0,T))}\leq C_0,~~{\rm for~all}~|\mu|>\mu_0.
\end{equation}
\end{lemma}
\begin{proof}
For $\xi=0$ this result has been proved  in \cite[Proposition 1]{SS19}. It can be easily extended to $\xi\in[-\pi/T,\pi/T)$  by replacing the (spatial) Fourier frequency $k$ in their expression of the linear operator by $k+\xi$.
\end{proof}

Combining the result in this lemma with the spectral properties in Lemma~\ref{L:spec_prep}, we obtain the following estimates for the Bloch semigroups $e^{\mathcal A_\xi[\phi]t}$.

\begin{proposition}\label{P:exp_decay}
  Suppose $\phi\in H^1_{\rm loc}(\RM)$ is a diffusively spectrally stable $T$-periodic stationary solution of the LLE \eqref{e:LLE}.
  Then the following properties hold.
  \begin{enumerate}
  \item For any fixed $\xi_0\in(0,\pi/T)$, there exist positive constants $C_0$ and $\eta_0$  such that
    \[
\left\|e^{\mathcal{A}_\xi[\phi] t}\right\|_{\mathcal L(L^2_{\rm per}(0,T))}\leq C_0e^{-\eta_0t}, 
\]
for all $t\geq0$ and all $\xi\in[-\pi/T,\pi/T)$ with $|\xi|>\xi_0$.
  \item With $\xi_1$ chosen as in  Lemma~\ref{L:spec_prep}~(ii), there exist positive constants  $C_1$ and $\eta_1$ such that for any $|\xi|<\xi_1$, if $\Pi(\xi)$ is the spectral projection onto the (one-dimensional) eigenspace associated to the eigenvalue  $\lambda_c(\xi)$  given by Lemma~\ref{L:spec_prep}~(ii), then 
\[
\left\|e^{\mathcal{A}_\xi[\phi]t}\left(I-\Pi(\xi)\right)\right\|_{\mathcal L(L^2_{\rm per}(0,T))}\leq C_1 e^{-\eta_1t},
\]
for all $t\geq0$. 
  \end{enumerate}
\end{proposition}

\begin{proof}
We use the Gearhart-Pr\"uss theorem\footnote{Technically, we are using a slight extension of the Gearhart-Pr\"uss theorem, requiring uniform boundedness of the resolvent
on the imaginary axis as opposed to a half plane.  See \cite[Corollary 2]{GS20} for a proof of this extension.} to prove the result. The absence of purely imaginary spectrum for the Bloch operators $\mathcal{A}_\xi[\phi]$ for $\xi\not=0$, implies that the resolvent estimate \eqref{e:resest} holds for any purely imaginary number $i\mu$, uniformly for all $\xi\in[-\pi/T,\pi/T)$ with $|\xi|>\xi_0$, for some fixed $\xi_0\in(0,\pi/T)$. The Gerhart-Pr\"uss theorem then implies the result in the first part of the lemma. The second part follows in the same way, noting that the operator  $\mathcal{A}_\xi[\phi](I-\Pi(\xi))$ has no purely imaginary spectrum for any $\xi\in[-\pi/T,\pi/T)$.
\end{proof}

\begin{remark}\label{r:subharmonic}
  \begin{enumerate}
  \item In the context of the LLE, the resolvent estimate \eqref{e:resest} actually holds for any complex number $\lambda=\delta+i\mu$ with $\delta>-1$ and $|\mu|>\mu_\delta$, for some $\mu_\delta>0$, which is the analogue for $\xi\not=0$ of the result obtained in \cite[Proposition 1]{SS19} for $\xi=0$. As a consequence, one can better characterize the exponential rates of decay in Proposition 2.6 by showing that
    \[
0<\eta_0<-\max\left\{\Re\sigma(\mathcal A_\xi[\phi]);\ \xi\in[-\pi/T,\pi/T),\ |\xi|>\xi_0\right\} ,
  \]
  and
    \[
0<\eta_1<-\max\left\{\Re\left(\sigma(\mathcal A_\xi[\phi])\setminus\{\lambda_c(\xi)\}\right);\ |\xi|<\xi_1\right\}.
    \]    
For our purposes, however, we do not need this more precise characterization.   
\item  Together with the Floquet-Bloch theory for subharmonic perturbations, the result in Proposition~\ref{P:exp_decay} gives the following estimate for the $C^0$-semigroup generated by $\mathcal A[\phi]$ when acting in~$L^2_{\rm per}(0,NT)$,
\[
\left\|e^{\mathcal{A}[\phi]t}(1-\mathcal{P}_{0,N})f\right\|_{L^2_{\rm per}(0,NT)}\leq C_N e^{-d_Nt},
\]
for all $t\geq0$, where $\mathcal{P}_{0,N}$ is the spectral projection of $L^2_{\rm per}(0,NT)$ onto the $NT$-periodic kernel of $\mathcal{A}[\phi]$ in Theorem~\ref{T:sub_main}. The above remark implies that $d_N\in(0,\delta_N)$, where $\delta_N$ is given by~\eqref{e:deltaN}. This is the exponential decay rate for the semigroup generated by $\mathcal{A}[\phi]$ mentioned in the introduction and required to prove the result in Theorem~\ref{T:SS}.  
\end{enumerate}
  \end{remark}

\section{Linear Asymptotic Modulational Stability to Localized Perturbations}\label{S:stab_loc}

Recall from the introduction that our goal is to obtain uniform rates of decay for subharmonic perturbations of a given $T$-periodic stationary solution of \eqref{e:LLE}.  
Our argument, which will be presented in Section \ref{S:stab_per} below,
is largely motivated by and modeled after the stability analysis to perturbations that are localized on $\RM$. 
As such, we first consider the case of localized perturbations
and present the proof of Theorem \ref{t:stab}.  We emphasize that the proof of Theorem \ref{t:stab} follows the general methodology introduced in \cite{JNRZ_Invent}
for the stability of periodic waves in general conservation or balance laws to classes of both localized and non-localized\footnote{In \cite{JNRZ_Invent} the authors
results cover the case when the perturbation slightly changes the phase at infinity of the underlying wave.  We do not consider such an extension here.} 
perturbations.  Nevertheless, we briefly outline the argument, for motivational purposes, as it applies to \eqref{e:LLE}.

Following \cite{JNRZ_Invent}, the general strategy of the proof of Theorem \ref{t:stab} is to use the Bloch transform to obtain estimates on the semigroup $e^{\mathcal{A}[\phi]t}$ 
from estimates on the Bloch semigroups $e^{\mathcal{A}_\xi[\phi]t}$ in Proposition \ref{P:exp_decay}.  
To this end, our goal is to decompose the semigroup $e^{\mathcal{A}[\phi]t}$ as
\[
e^{\mathcal{A}[\phi]t}=``{\rm Critical~Part}"~+~``{\rm Exponentially~Damped~Part}",
\]
where, owing to Lemma \ref{L:spec_prep}, the ``critical part'' should be dominated by the translation mode $\phi'$.  
Note throughout this section, we adopt the notation
$A\lesssim B$ to mean there exists a constant $C>0$ such that $A\leq CB$.

Let $v\in L^1(\RM)\cap L^2(\RM)$.  We begin by decomposing $e^{\mathcal{A}[\phi]t}$
into low-frequency and high-frequency components: with $\xi_1$ defined as in Proposition \ref{P:exp_decay}, let $\rho$ be a smooth cutoff function with $\rho(\xi)=1$
for $|\xi|<\xi_1/2$ and $\rho(\xi)=0$ for $|\xi|>\xi_1$ and use \eqref{e:blochsoln} to write
\begin{equation}\label{e:loc_shf}
\begin{aligned}
e^{\mathcal{A}[\phi]t}v(x)&=\frac{1}{2\pi}\int_{-\pi/T}^{\pi/T} \rho(\xi)e^{i\xi x}e^{\mathcal{A}_\xi[\phi] t}\check{v}(\xi,x)d\xi+
	\frac{1}{2\pi}\int_{-\pi/T}^{\pi/T} (1-\rho(\xi))e^{i\xi x}e^{\mathcal{A}_\xi[\phi] t}\check{v}(\xi,x)d\xi\\
&=:S_{lf}(t)v(x)+S_{hf}(t)v(x),
\end{aligned}
\end{equation}
where $S_{lf}$ and $S_{hf}$ denote the low- and high-frequency components of the solution operator $e^{\mathcal{A}[\phi]t}$, respectively.
According to Lemma \ref{L:spec_prep}, the spectrum of the Bloch operators $\mathcal{A}_\xi[\phi]$ have a uniform spectral gap on the support
of $(1-\rho(\xi))$ and hence, by Parseval's identity \eqref{e:parseval_loc} and Proposition \ref{P:exp_decay}, there exists a constant $\eta>0$ such that
\begin{equation}\label{e:loc_hf_est}
\begin{aligned}
\left\|S_{hf}(t)v\right\|_{L^2(\RM)}^2&=\frac{1}{2\pi T}\int_{-\pi/T}^{\pi/T}\left\|(1-\rho(\xi))e^{\mathcal{A}_\xi[\phi]t}\check{v}(\xi,x)\right\|_{L^2(0,T)}^2d\xi
\lesssim e^{-\eta t}\|v\|_{L^2(\RM)}
\end{aligned}
\end{equation}
valid for all $v\in L^2(\RM)$.  It thus remains to study the low-frequency component of the solution operator.

To this end, we further decompose $S_{lf}(t)$ into the contribution from the critical mode near $(\lambda,\xi)=(0,0)$ and the contribution
from the low-frequency spectrum bounded away from $\lambda=0$.  Accordingly, for each $|\xi|<\xi_1$ 
let $\Pi(\xi)$ be the spectral projection onto the critical mode of $\mathcal{A}_\xi[\phi]$ as defined in Proposition \ref{P:exp_decay}, and
note it is given explicitly  via
\begin{equation}\label{e:spec_proj}
\Pi(\xi)g(x) = \left<\widetilde{\Phi}_\xi,g\right>_{L^2(0,T)}\Phi_\xi(x),
\end{equation}
where $\widetilde{\Phi}_\xi$ is the element in the kernel of the adjoint $\mathcal{A}_\xi[\phi]^*-\overline{\lambda_c(\xi)}I$ satisfying 
$\left<\widetilde{\Phi}_\xi,\Phi_\xi\right>_{L^2(0,T)}=1$.  We can thus decompose $S_{lf}(t)$ as
\begin{align}
S_{lf}(t)v(x)&=\frac{1}{2\pi}\int_{-\pi/T}^{\pi/T} \rho(\xi)e^{i\xi x}e^{\mathcal{A}_\xi[\phi] t}\Pi(\xi)\check{v}(\xi,x)d\xi
	+\frac{1}{2\pi}\int_{-\pi/T}^{\pi/T} \rho(\xi)e^{i\xi x}e^{\mathcal{A}_\xi[\phi] t}(1-\Pi(\xi))\check{v}(\xi,x)d\xi \nonumber\\
&=:S_c(t)v(x)+\widetilde{S}_{lf}(t)v(x).\label{e:loc_slf}
\end{align}
Using Parseval's identity \eqref{e:parseval_loc} and Proposition \ref{P:exp_decay} again, it follows that, by possibly choosing $\eta>0$ above smaller, 
\begin{equation}\label{e:loc_lf_expest}
\left\|\widetilde{S}_{lf}(t)v\right\|_{L^2(\RM)}\lesssim e^{-\eta t}\|v\|_{L^2(\RM)}
\end{equation}
valid for all $v\in L^2(\RM)$.

For the critical component $S_c(t)$ of the solution operator, 
note that for any $v\in L^2(\RM)$ we have
\begin{align*}
e^{\mathcal{A}_\xi[\phi]t}\Pi(\xi)\check{v}(\xi,x)&=e^{\lambda_c(\xi)t}\Phi_\xi(x)\left<\widetilde{\Phi}_\xi,\check{v}(\xi,\cdot)\right>_{L^2(0,T)}\\
&=e^{\lambda_c(\xi)t}\left(\phi'+i\xi\left(\frac{\Phi_\xi(x)-\phi'(x)}{i\xi}\right)\right)\left<\widetilde{\Phi}_\xi,\check{v}(\xi,\cdot)\right>_{L^2(0,T)},
\end{align*}
where we note by Lemma \ref{L:spec_prep} that
\begin{equation}\label{e:diff_quotient_bd}
\sup_{x\in\RM}\left(\sup_{\xi\in[-\pi/T,\pi/T)}\left|\rho(\xi)\left(\frac{\Phi_\xi(x)-\phi'(x)}{i\xi}\right)\right|\right)\lesssim 1.
\end{equation}
We may thus decompose $S_c$ further as
\begin{equation}\label{phase}
\begin{aligned}
S_{c}(t)v(x)&=\phi'(x)\frac{1}{2\pi}\int_{-\pi/T}^{\pi/T} \rho(\xi)e^{i\xi x+\lambda_c(\xi)t}\left<\widetilde{\Phi}_\xi,\check{v}(\xi,\cdot)\right>_{L^2(0,T)}d\xi\\
&\qquad+\frac{1}{2\pi}\int_{-\pi/T}^{\pi/T}\rho(\xi)e^{i\xi x+\lambda_c(\xi)t}i\xi\left(\frac{\Phi_\xi(x)-\phi'(x)}{i\xi}\right)\left<\widetilde{\Phi}_\xi,\check{v}(\xi,\cdot)\right>_{L^2(0,T)}d\xi\\
&=\phi'(x)s_p(t)v(x)+\widetilde{S}_{c}(t)v(x).
\end{aligned}
\end{equation}
Now, observe that by definition (see \eqref{e:Bloch}) we have
\begin{align*}
    \LA\widetilde{\Phi}_\xi, \check{v}(\xi,\cdot)\RA_{L^2(0,T)} &= \int_0^T \overline{\widetilde{\Phi}_\xi(x)} \sum_{\ell\in\Z} e^{2\pi i\ell x/T}\hat{v}(\xi+2\pi \ell /T)\, dx\\
    &= \sum_{\ell\in\Z}\hat{v}(\xi+2\pi\ell/T)\int_0^T\overline{\widetilde{\Phi}_\xi(x)} e^{2\pi i\ell x/T}dx\\
    &=\sum_{\ell\in\Z}\hat{v}(\xi+2\pi\ell/T) \overline{\widehat{\widetilde{\Phi}}_\xi(2\pi\ell/T)}
\end{align*}
and hence, since $|\hat{v}(z)|\leq \|v\|_{L^1(\RM)}$ for all $z\in\RM$, we have
\begin{align}
    \rho(\xi) \left|\LA\widetilde{\Phi}_\xi, \check{v}(\xi,\cdot)\RA_{L^2(0,T)}\right|^2 &\leq \rho(\xi) \left(\sum_{\ell\in\Z}|\hat{v}(\xi+2\pi\ell/T)| \left|\widehat{\widetilde{\Phi}}_\xi(2\pi\ell/T)\right|\right)^2\nonumber\\
    &\leq \rho(\xi)\|v\|_{L^1(\RM)}^2\left(\sum_{\ell\in\Z} (1+|\ell|^2)^{1/2}\left|\widehat{\widetilde{\Phi}}_\xi(2\pi\ell/T)\right|(1+|\ell|^2)^{-1/2}\right)^2 \label{e:tech_ineq}\\
& \lesssim\|v\|_{L^1(\RM)}^2 \sup_{\xi\in[-\pi/T,\pi/T)}\left(\rho(\xi)\left\|\widetilde{\Phi}_\xi\right\|_{H^1_{\rm per}(0,T)}^2\right),\nonumber
\end{align}
where the final inequality follows by the Cauchy-Schwarz inequality.  Using Parseval's identity \eqref{e:parseval_loc} it follows that the phase shift component of $S_c$ satisfies
\begin{equation}\label{e:loc_sp_bd}
\begin{aligned}
\left\|s_p(t)v\right\|_{L^2(\RM)} &= \left(\frac{1}{2\pi T}\int_{-\pi/T}^{\pi/T}\left\|\rho(\xi)e^{\lambda_c(\xi)t}\left<\widetilde{\Phi}_\xi,\check{v}(\xi,\cdot)\right>_{L^2(0,T)}
		\right\|_{L^2(0,T)}^2d\xi\right)^{1/2}\\
&\lesssim 	\|e^{-d\xi^2 t}\|_{L^2_\xi(\RM)}\|v\|_{L^1(\RM)}\\
&\lesssim (1+t)^{-1/4}\|v\|_{L^1(\RM)}
\end{aligned}
\end{equation}
and, similarly,
\begin{equation}\label{e:loc_sc_bd}
\left\|\widetilde{S}_{c}(t)v\right\|_{L^2(\RM)}\lesssim \|\xi e^{-d\xi^2 t}\|_{L^2_\xi(\RM)}\|v\|_{L^1(\RM)}\lesssim (1+t)^{-3/4}\|v\|_{L^1(\RM)},
\end{equation}
where the bounds on $\|e^{-d\xi^2 t}\|_{L^2_\xi(\RM)}$ and $\|\xi e^{-d\xi^2 t}\|_{L^2_\xi(\RM)}$ follow from an elementary scaling analysis.

In summary, for each $v\in L^1(\RM)\cap L^2(\RM)$ we have the decomposition
\begin{equation}\label{e:final_decomp_loc}
e^{\mathcal{A}[\phi]t}v(x)=\phi'(x)s_p(t)v(x)+\widetilde{S}_c(t)v(x)+\widetilde{S}_{lf}(t)v(x)+S_{hf}(t)v(x),
\end{equation}
where the operators $s_p(t)$ and $\widetilde{S}_c(t)$ are defined in \eqref{phase}, and where the operators  $\widetilde{S}_{lf}(t)$ and $S_{hf}(t)$ are defined in
\eqref{e:loc_slf} and \eqref{e:loc_shf}, respectively.  Recalling the estimates \eqref{e:loc_sp_bd}-\eqref{e:loc_sc_bd},
valid for all $t\geq 0$, 
as well as the exponential decay estimates \eqref{e:loc_lf_expest} and \eqref{e:loc_hf_est} on $\widetilde{S}_{lf}(t)$ and $S_{hf}(t)$, 
respectively, the proof of Theorem \ref{t:stab} follows by setting $\gamma(x,t):=\left(s_p(t)v\right)(x)$.

\begin{remark}\label{R:diffusive_decay}
Note that to obtain $L^2\to L^2$ bounds on $s_p$ and $\widetilde{S}_c$ above, one would use in place of \eqref{e:tech_ineq} the slightly sharper
estimate
\[
\rho(\xi)\left|\left<\widetilde{\Phi}_\xi,\check{v}(\xi,\cdot)\right>_{L^2(0,T)}\right|^2\leq\rho(\xi)\left\|\widetilde{\Phi}_\xi\right\|_{L^2(0,T)}^2\|\check{v}(\xi,\cdot)\|_{L^2(0,T)}^2
\]
which, by Parseval, would then lead to the bounds
\[
\|s_p(t)v\|_{L^2(\RM)}\lesssim\|e^{-d\xi^2 t}\|_{L^\infty_\xi(\RM)}\|v\|_{L^2(\RM)}\lesssim \|v\|_{L^2(\RM)}
\]
and
\[
\|\widetilde{S}_c(t)v\|_{L^2(\RM)}\lesssim\|\xi e^{-d\xi^2 t}\|_{L^\infty_\xi(\RM)}\|v\|_{L^2(\RM)}\lesssim (1+t)^{-1/2}\|v\|_{L^2(\RM)}
\]
In particular, since $s_p(t)$ does not exhibit decay from $L^2$ to $L^2$, the final decomposition \eqref{e:final_decomp_loc} implies only a  bounded linear
stability from $L^2$ to $L^2$.    The faster polynomial rates of decay in \eqref{e:loc_sp_bd}-\eqref{e:loc_sc_bd} rely on being able to control
the initial perturbation in $L^1$ as well, and introduces diffusive rates of decay of perturbations.
\end{remark}

Finally, we end our study of the localized analysis by describing at a finer level the long-time dynamics of the modulation function $\gamma$ in Theorem \ref{t:stab}.  
Note that from the explicit form of the phase operator $s_p(t)$ defined in \eqref{phase} it is natural to expect that for a given $v\in L^2(\RM)$
the function $s_p(t)v$ should be well approximated (for at least large time) by the function
\begin{equation}\label{w}
 (w(t)v)(x):=\frac{1}{2\pi}\int_{-\pi/T}^{\pi/T} \rho(\xi)e^{i\xi x+(ia\xi-d\xi^2)t}\left<\widetilde{\Phi}_\xi,\check{v}(\xi,\cdot)\right>_{L^2(0,T)}d\xi.
\end{equation}
Precisely, following the techniques from above we have the bound
\begin{align*}
\left\|s_p(t)v-w(t)v\right\|_{L^2(\RM)}&=
\left\|\frac{1}{2\pi}\int_{-\pi/T}^{\pi/T}e^{i\xi x}\rho(\xi)\left(e^{\lambda_c(\xi) t}-e^{(ia\xi-d\xi^2) t}\right)\left<\widetilde\Phi_\xi,\check{v}(\xi,\cdot)\right>_{L^2(0,T)}d\xi\right\|_{L^2(\RM)}\\
&\lesssim\left\|e^{-d\xi^2 t}\left(e^{(\lambda_c(\xi)-(ia\xi-d\xi^2))t}-1\right)\right\|_{L^2_\xi(\RM)}\|v\|_{L^1(\RM)}\\
&\lesssim\left\|e^{-d\xi^2 t}\xi^3 t\right\|_{L^2_\xi(\RM)}\|v\|_{L^1(\RM)}\\
&\lesssim(1+t)^{-3/4}\|v\|_{L^1(\RM)}.
\end{align*}
Noting that the above $w(x,t):=(w(t)v)(x)$ defined in \eqref{w} is the unique solution of the linear diffusion IVP 
\begin{equation}\label{e:whitham_ivp}
\left\{\begin{aligned}
 w_t&=aw_x+dw_{xx}\\
w(x,0)&=s_p(0)v(x)=\frac{1}{2\pi}\int_{-\pi/T}^{\pi/T} \rho(\xi)e^{i\xi x}\left<\widetilde{\Phi}_\xi,\check{v}(\xi,\cdot)\right>_{L^2(0,T)}d\xi.
\end{aligned}\right.
\end{equation}
posed on $L^2(\RM)$, this establishes the following result concerning the behavior of the modulation function $\gamma(x,t)$ 
from Theorem \ref{t:stab}.

\begin{corollary}[Asymptotic Modulational Behavior]\label{c:behavior}
Under the hypotheses of Theorem \ref{t:stab}, for each $v\in L^1(\RM)\cap L^2(\RM)$ the modulation function $\gamma$ satisfies
the estimate
\[
\left\|\gamma(\cdot,t)-w(\cdot,t)\right\|_{L^2(\RM)}\lesssim(1+t)^{-3/4}\|v\|_{L^1(\RM)\cap L^2(\RM)}
\]
valid or all $t>0$, where here $w$ is the solution to the linear heat equation
\[
 w_t=-i\lambda_c'(0)w_{x}-\frac{1}{2}\lambda''_c(0)w_{xx},~~x\in\RM,~~t>0
\]
with initial data prescribed as in \eqref{e:whitham_ivp}.
\end{corollary}

\begin{remark}
To place the final calculations above in a broader context, we note that the PDE in \eqref{e:whitham_ivp} corresponds to the Whitham modulation equation associated
to \eqref{e:LLE}.  Following the work in \cite[Appendix B]{JNRZ_Invent} it is possible to show through a formal multiple scales analysis 
that an approximate solution to the profile equation \eqref{e:profile} is given by
\[
\psi(x,t)\approx \widetilde{\psi}\left(\Psi(x,t)\right)
\]
where the wave number $k:=\Psi_x$ satisfies the Whitham equation
\[
k_t=-i\lambda_c'(0)k_x-\frac{1}{2}\lambda_c''(0)k_{xx}.
\]
This suggests that the large time modulational behavior should be governed by a solutions of the heat equation, which is precisely what is made
rigorous through Corollary \ref{c:behavior}.  For more information on dynamical predictions of Whitham modulation equations,
see \cite{DSSS,JNRZ_Invent}.
\end{remark}

\section{Linear Asymptotic Modulational Stability to Subharmonic Perturbations}\label{S:stab_per}

Motivated by the analysis in Section \ref{S:stab_loc}, we now strive to obtain decay rates on the semigroup $e^{\mathcal{A}[\phi]t}$ acting on classes of subharmonic
perturbations which are uniform in $N$.   As we will see, the analysis is based on a decomposition of the solution operator $e^{\mathcal{A}[\phi]t}$ which is largely motivated
by the analysis in Section \ref{S:stab_loc}.  One key difference, however, is that in the subharmonic case, the eigenvalue $\lambda=0$ is isolated
from the remaining $NT$-periodic eigenvalues of $\mathcal{A}[\phi]$, which leads to a slightly different decomposition of the solution operator near $(\lambda,\xi)=(0,0)$
than used in Section \ref{S:stab_loc} above.

For each $N\in \NM$ and each $p\geq 1$, we set, for notational convenience, 
\[
L^p_N:=L^p_{\rm per}(0,NT).
\]
Further, throughout the remainder of our work, we will use the previously introduced notation $A\lesssim B$ to indicate there exists a constant $C>0$ which is \emph{independent of
$N$} such that $A\leq CB$.
Now, for fixed $N\in\NM$ let  $f\in L^2_N$  and recall from \eqref{e:per_semigrp} 
that the action of the semigroup $e^{\mathcal{A}[\phi]t}$ on $f$ can be represented through the use of the Bloch transform as\footnote{This representation
formula is the analogue of \eqref{e:blochsoln} used in Section \ref{S:stab_loc}.}
\[
e^{\mathcal{A}[\phi]t}f(x)=\frac{1}{NT}\sum_{\xi\in\Omega_N} e^{i\xi x} e^{\mathcal{A}_\xi[\phi] t}\mathcal{B}_T(f)(\xi,x),
\]
where $\Omega_N\subset [-\pi/T,\pi/T)$ is defined in \eqref{e:freq} as the set of Bloch frequencies corresponding to $NT$-periodic perturbations
and $\mathcal{B}_T(f)$, defined in \eqref{e:Bloch_Tper}, denotes the $T$-periodic Bloch transform of the $NT$-periodic function $f$.

Following the general procedure in Section \ref{S:stab_loc}, we begin by decomposing the
subharmonic solution operator 
into low-frequency and high-frequency parts.  Letting $\rho$ be a smooth cut-off function as in Section \ref{S:stab_loc}, note for all $f\in L^2_N$ we have
\begin{align}
e^{\mathcal{A}[\phi]t}f(x)&=\frac{1}{NT}\sum_{\xi\in\Omega_N}\rho(\xi) e^{i\xi x} e^{\mathcal{A}_\xi[\phi] t}\mathcal{B}_T(f)(\xi,x)
	+\frac{1}{NT}\sum_{\xi\in\Omega_N}\left(1-\rho(\xi)\right) e^{i\xi x} e^{\mathcal{A}_\xi[\phi] t}\mathcal{B}_T(f)(\xi,x) \nonumber\\
&=:S_{lf,N}(t)f(x)+S_{hf,N}(t)f(x).\label{e:shf}
\end{align}
To estimate the high-frequency component, we use the discrete Parseval identity \eqref{e:parseval_per} to get
\begin{align*}
\|S_{hf,N}(t)f\|_{L_N^2}^2 &= \frac{1}{NT^2} \sum_{\xi\in\Omega_N}\|(1-\rho(\xi)) e^{A_\xi[\phi] t} \mathcal{B}_T(f)(\xi,\cdot)\|_{L^2(0,T)}^2\\
&\leq \frac{1}{NT^2} \sum_{\xi\in\Omega_N}(1-\rho(\xi))^2 \|e^{A_\xi[\phi] t}\|_{\mathcal{L}(L_{\text{per}}^2(0,T))}^2
\|\mathcal{B}_T(f)(\xi,\cdot)\|_{L^2(0,T)}^2.
\end{align*}
From Proposition \ref{P:exp_decay}, it follows from Lemma \ref{L:spec_prep} and Proposition \ref{P:exp_decay} that there exists a constant
$\eta>0$ such that
\[
\max_{\xi\in\Omega_N}(1-\rho(\xi))^2\|e^{A_\xi [\phi]t}\|_{\mathcal{L}(L_{\text{per}}^2(0,T))}
\lesssim e^{-\eta t}
\]
yielding the exponential decay estimate
\begin{equation}\label{e:shf_bd}
\|S_{hf,N}(t)f\|_{L_N^2} \lesssim e^{-\eta t} \left(\frac{1}{NT^2}\sum_{\xi\in\Omega_N}\|\mathcal{B}_T(f)(\xi,\cdot)\|_{L^2(0,T)}^2\right)^{1/2} = e^{-\eta t}\|f\|_{L_N^2},
\end{equation}
where the last equality again follows by Parseval's identity \eqref{e:parseval_per}.

Continuing on, to study the low-frequency component, let 
$\Pi(\xi)$ be the rank-one spectral projection onto the critical mode of $\mathcal{A}_\xi[\phi]$ from Proposition \ref{P:exp_decay}
and note that $S_{lf,N}(t)$ can be 
further decomposed into the contribution from the critical mode near $(\lambda,\xi)=(0,0)$ and the contribution from the low-frequency spectrum bounded
away from $\lambda=0$ via 
\begin{equation}\label{e:slf}
\begin{aligned}
    S_{lf,N}(t)f(x) &= \frac{1}{NT}\sum_{\xi\in\Omega_N}\rho(\xi)e^{i\xi x} e^{A_\xi[\phi] t}\Pi(\xi) \mathcal{B}_T(f)(\xi,x)\\
    &\qquad\qquad + \frac{1}{NT}\sum_{\xi\in\Omega_N}\rho(\xi)e^{i\xi x} e^{A_\xi[\phi] t}(1-\Pi(\xi)) \mathcal{B}_T(f)(\xi,x)\\
    &=: S_{c,N}(t)f(x) + \widetilde{S}_{lf,N}(t)f(x).
\end{aligned}
\end{equation}
Using Parseval's identity \eqref{e:parseval_per} and Proposition \ref{P:exp_decay} we know 
by possibly choosing $\eta>0$ above even smaller, that
\begin{equation}\label{e:slf_bd}
\|\widetilde{S}_{lf,N}(t)f\|_{L_N^2} \lesssim e^{-\eta t}\|f\|_{L_N^2}.
\end{equation}
For the critical component of $S_{c,N}(t)$, 
 note that by Lemma \ref{L:per_bloch} the $\xi=0$ term\footnote{Observe the $\xi=0$ term, corresponding to the projection onto the kernel of $\mathcal{A}_0[\phi]$ does
not experience any type of temporal decay.  Consequently, we factor it out of the remaining sums which, as we will see, do decay in time.} can be identified as
\[
\frac{1}{NT}\Pi(0)\mathcal{B}_T(f)(0,x)=\frac{1}{NT}\phi'(x)\left<\widetilde{\Phi}_0,\mathcal{B}_T(f)(0,\cdot)\right>_{L^2(0,T)}=\frac{1}{N}\phi'(x)\left<\widetilde{\Phi}_0,f\right>_{L^2_N}.
\]
Since \eqref{e:spec_proj} and Lemma \ref{L:spec_prep} imply the projection of $L^2_N$ onto the $NT$-periodic 
kernel of $\mathcal{A}_0[\phi]$ is given explicitly by\footnote{Recall
the left $T$-periodic eigenfunction $\widetilde{\Phi}_0$ is normalized so that $\left<\widetilde{\Phi}_0,\phi'\right>_{L^2(0,T)}=1$, and hence $\LA \widetilde{\Phi}_0, \phi'\RA_{L^2_N}=N$.}
\begin{equation}\label{e:proj}
\mathcal{P}_{0,N}f(x)=\frac{1}{N}\phi'(x)\left<\widetilde{\Phi}_0,f\right>_{L^2_N},
\end{equation}
it follows from above that $S_{c,N}$ decomposes further as
\begin{equation}\label{e:sp}
\begin{aligned}
S_{c,N}f(x)&=e^{\mathcal{A}_0[\phi]t}\mathcal{P}_{0,N}f(x)
+\frac{1}{NT}\sum_{\xi\in\Omega_N\setminus\{0\}}\rho(\xi)e^{i\xi x} e^{\lambda_c(\xi) t}\Phi_\xi(x)\LA\widetilde{\Phi}_\xi, \mathcal{B}_T(f)(\xi,\cdot)\RA_{L^2(0,T)}\\
&=e^{\mathcal{A}[\phi]t}\mathcal{P}_{0,N}f(x)
+\phi'(x)\frac{1}{NT}\sum_{\xi\in\Omega_N\setminus\{0\}}\rho(\xi)e^{i\xi x} e^{\lambda_c(\xi) t}\LA\widetilde{\Phi}_\xi, \mathcal{B}_T(f)(\xi,\cdot)\RA_{L^2(0,T)}\\
    &\qquad + \frac{1}{NT}\sum_{\xi\in\Omega_N\setminus\{0\}}\rho(\xi)e^{i\xi x} i\xi\left(\frac{\Phi_\xi(x)-\phi'(x)}{i\xi}\right)e^{\lambda_c(\xi) t}\LA\widetilde{\Phi}_\xi, \mathcal{B}_T(f)(\xi,\cdot)\RA_{L^2(0,T)}\\
&=:e^{\mathcal{A}[\phi]t}\mathcal{P}_{0,N}f(x)+\phi'(x)s_{p,N}(t)f(x) + \widetilde{S}_{c,N}(t)f(x).
\end{aligned}
\end{equation}
Now, arguing as in \eqref{e:tech_ineq} we find that
\begin{align*}
    \rho(\xi) \left|\LA\widetilde{\Phi}_\xi, \mathcal{B}_T(f)(\xi,\cdot)\RA_{L^2(0,T)}\right|^2 
	&\lesssim\|f\|_{L_N^1}^2 \sup_{\xi\in[-\pi/T,\pi/T)}\left(\rho(\xi)\left\|\widetilde{\Phi}_\xi\right\|_{H^1_{\rm per}(0,T)}^2\right),\nonumber
\end{align*}
and hence, recalling the bound \eqref{e:diff_quotient_bd} and that
\[
\left|\rho(\xi)^{1/2}e^{\lambda_c(\xi)t}\right|\lesssim e^{-d\xi^2 t}
\]
for some constant $d>0$,  it follows by Parseval's identity \eqref{e:parseval_per} that
\begin{align}
\|\widetilde{S}_{c,N}(t)f\|_{L_N^2}^2 
&=\frac{1}{NT^2}\sum_{\xi\in\Omega_N}\left\|\rho(\xi) i\xi\left(\frac{\Phi_\xi(x)-\phi'(x)}{i\xi}\right)e^{\lambda_c(\xi) t}\LA\widetilde{\Phi}_\xi, \mathcal{B}_T(f)(\xi,\cdot)\RA_{L^2(0,T)}\right\|_{L^2_x(0,T)}^2 \nonumber\\
&\lesssim \left(\frac{1}{N}\sum_{\xi\in\Omega_N}\xi^2e^{-2d\xi^2 t}\right)\|f\|_{L_N^1}^2.\label{e:sc_bd}
\end{align}
Furthermore, following similar steps as above we have the bound
\begin{equation}\label{e:sp_bd}
\left\|s_{p,N}(t)f\right\|_{L^2_N}^2\lesssim\left(\frac{1}{N}\sum_{\xi\in\Omega_N\setminus\{0\}} e^{-2d\xi^2t}\right)\|f\|_{L^1_N}^2.
\end{equation}

To complete the proof of Theorem \ref{T:sub_main}, it remains to study the finite sums
\begin{equation}\label{sums_goal}
\frac{1}{N}\sum_{\xi_j\in\Omega_N\setminus\{0\}}e^{-2d\xi_j^2 t}~~{\rm and}~~\frac{1}{N}\sum_{\xi_j\in\Omega_N}\xi_j^2e^{-2d\xi_j^2 t}.
\end{equation} 
To obtain uniform (in $N$) rates of decay, note that 
both of the sums in \eqref{sums_goal} correspond to Riemann sum approximations to  the integrals\footnote{Recalling $\Delta\xi_j=2\pi/NT$, we see that we technically need to multiply and divide by a harmless factor of $2\pi/T$.}
\begin{equation}\label{e:estint}
\int_{-\pi/T}^{\pi/T} e^{-2d\xi^2t}d\xi~~{\rm and}~~\int_{-\pi/T}^{\pi/T} \xi^2 e^{-2d\xi^2t}d\xi,
\end{equation}
which, through elementary scaling analysis (as in the previous section), decay like $(1+t)^{-1/2}$ and $(1+t)^{-3/2}$, respectively.
The next result uses this observation to obtain analogous decay bounds on the discrete sums in \eqref{sums_goal}.

\begin{proposition}\label{P:Uniform_Bounds}
There exists a constant $C>0$ such that for all  $N\in\NM$ and $t>0$ we have
\begin{equation}\label{est1}
\frac{2\pi}{NT}\sum_{\xi_j\in\Omega_N\setminus\{0\}}e^{-2d\xi_j^2 t}\leq C(1+t)^{-1/2}
\end{equation}
and
\begin{equation}\label{est2}
\frac{2\pi}{NT}\sum_{\xi_j\in\Omega_N}\xi_j^2e^{-2d\xi_j^2 t} \leq C (1+t)^{-3/2}.
\end{equation}
\end{proposition}

\begin{remark}
The above bounds show that the polynomial decay rates on localized perturbations obtained in Theorem \ref{t:stab} provide uniform (in $N$) upper bounds on the 
decay rates associated to subharmonic perturbations, i.e. to perturbations with period $NT$ for some $N\in\NM$.  In the next section, we sharpen these estimates
and provide associated lower bounds on the subharmonic decay rates, showing that, in fact, the localized decay rates provide \emph{sharp} uniform bounds
on subharmonic perturbations.  The proof of the upper bounds in Proposition \ref{P:Uniform_Bounds}, however, is based on a substantially simpler monotonicity argument, which we now provide.
\end{remark}

\begin{proof}[Proof of Proposition \ref{P:Uniform_Bounds}.]
  We compare the sums in \eqref{est1} and \eqref{est2} with the integrals in \eqref{e:estint}.
For \eqref{est1}, notice that, for each $t>0$, the function $\xi\mapsto e^{-2d\xi^2t}$ is monotonically decreasing for $\xi>0$ and monotonically increasing for $\xi<0$. Together with the equality $(\xi_{j}-\xi_{j-1})=2\pi/NT$ these monotonicity property imply that
\begin{equation}\label{est1b}
    \begin{aligned}
    \frac{2\pi}{NT}\sum_{\xi_j\in\Omega_N\setminus\{0\}}e^{-2d\xi_j^2 t} &=
\underset{\xi_j<0}{\sum_{\xi_j\in\Omega_N}}e^{-2d\xi_j^2 t}(\xi_{j+1}-\xi_j) +
\underset{\xi_j>0}{\sum_{\xi_j\in\Omega_N}}e^{-2d\xi_j^2 t}(\xi_{j}-\xi_{j-1})\\ 
&\leq \underset{\xi_j<0}{\sum_{\xi_j\in\Omega_N}}\int_{\xi_j}^{\xi_{j+1}}e^{-2d\xi^2 t}d\xi +
\underset{\xi_j>0}{\sum_{\xi_j\in\Omega_N}}\int_{\xi_{j-1}}^{\xi_{j}}e^{-2d\xi^2 t}d\xi \\ 
&\leq \int_{-\pi/T}^{\pi/T} e^{-2d\xi^2t}d\xi,
    \end{aligned}
\end{equation}
which proves the inequality \eqref{est1}.

For \eqref{est2}, we have to slightly modify this argument because the function $\xi\mapsto \xi^2e^{-2d\xi^2t}$ does not have the same monotonicity properties. Indeed, it has a global minimum at $\xi=0$ and two  global maxima at $\xi=\pm R$ where 
\[
R=(2dt)^{-1/2}\quad {\rm and}\quad R^2e^{-2dR^2t}=(2dte)^{-1};
\]
see Figure \ref{F:uppdrbd}.  
\begin{figure}[t]
\begin{center}
\includegraphics[scale=0.55]{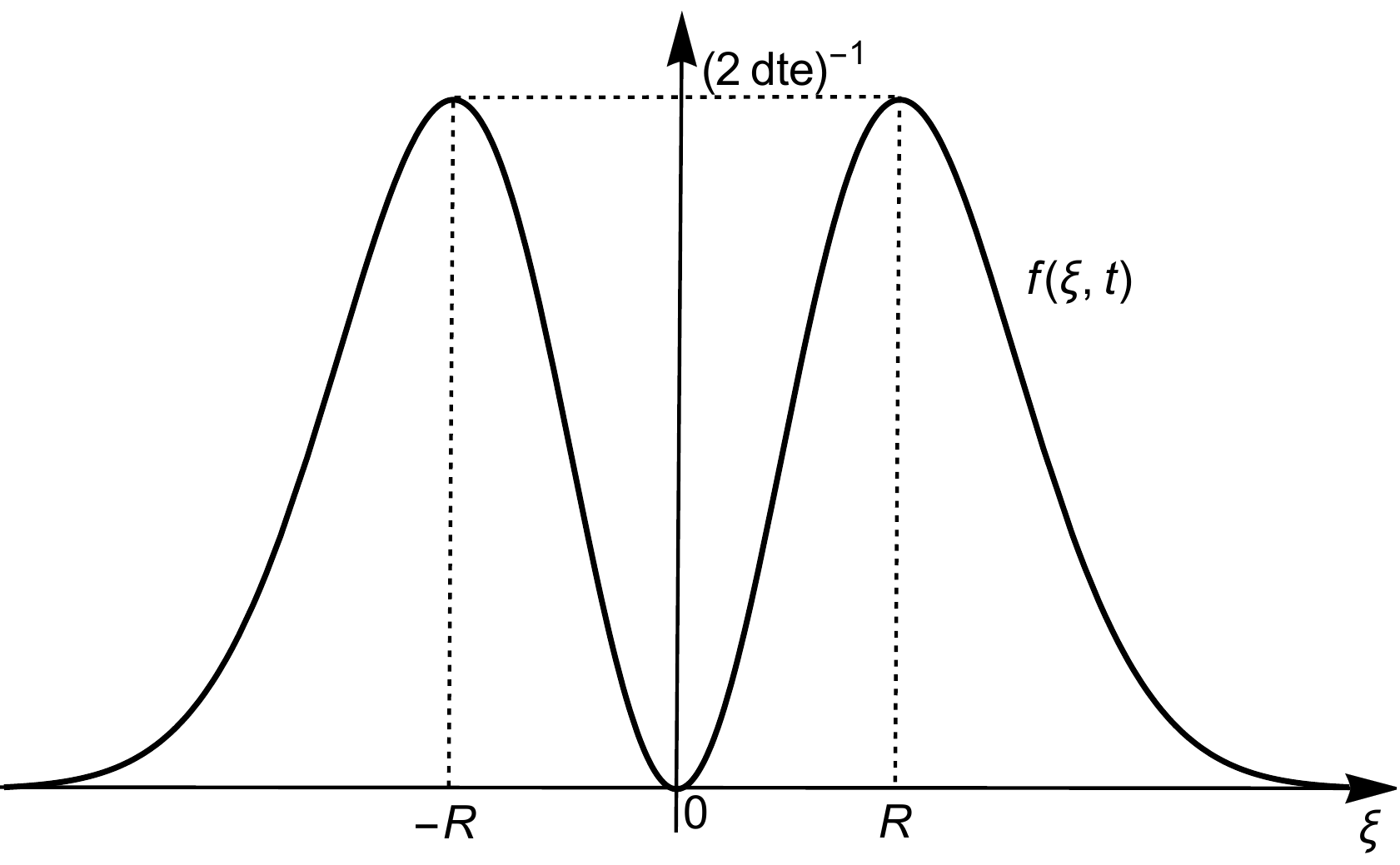}
\caption{A schematic drawing of the function $\xi\mapsto \xi^2e^{-2d\xi^2t}$ used in the proof of Proposition \ref{P:Uniform_Bounds}.  Note that the area under
the supremum between $-R\leq\xi\leq R$ is  $2R(2dte)^{-1}=2e^{-1}(2dt)^{-3/2}$.}\label{F:uppdrbd}
\end{center}
\end{figure}
Then for $0<t\leq T^2/2d\pi^2$, the values $\pm R$ do not belong to the interval $(-\pi/T,\pi/T)$ and we can easily estimate the sum in \eqref{est2},
\begin{equation}\label{est2b}
\frac{2\pi}{NT}\sum_{\xi_j\in\Omega_N}\xi_j^2e^{-2d\xi_j^2 t} \leq 
\frac{2\pi}{NT}\sum_{\xi_j\in\Omega_N} \frac{\pi^2}{T^2} \leq \frac{2\pi^3}{T^3}.
\end{equation}
For $t> T^2/2d\pi^2$, we consider the function
\[
G_t(\xi):=\begin{cases}
(2dte)^{-1}, & |\xi|\leq R\\
\xi^2e^{-2d\xi^2t}, & R<|\xi|\leq\pi/T
\end{cases}
\]
which is nonincreasing for $\xi>0$ and nondecreasing for $\xi<0$. Then by arguing as for \eqref{est1b}, we find
\begin{align*}
    \frac{2\pi}{NT}\sum_{\xi_j\in\Omega_N\setminus\{0\}}\xi_j^2e^{-2d\xi_j^2 t} &=
\underset{\xi_j<0}{\sum_{\xi_j\in\Omega_N}}\xi_j^2e^{-2d\xi_j^2 t}(\xi_{j+1}-\xi_j) +
\underset{\xi_j>0}{\sum_{\xi_j\in\Omega_N}}\xi_j^2e^{-2d\xi_j^2 t}(\xi_{j}-\xi_{j-1})\\
&\leq \underset{\xi_j<0}{\sum_{\xi_j\in\Omega_N}}\int_{\xi_j}^{\xi_{j+1}} G_t(\xi)d\xi +
\underset{\xi_j>0}{\sum_{\xi_j\in\Omega_N}}\int_{\xi_{j-1}}^{\xi_{j}}G_t(\xi)d\xi \\
&\leq \int_{-\pi/T}^{\pi/T} G_t(\xi)d\xi\ \leq\ 2e^{-1}(2dt)^{-3/2}+ \int_{-\pi/T}^{\pi/T}\xi^2 e^{-2d\xi^2t}d\xi.
\end{align*}
Together with \eqref{est2b} this proves the inequality \eqref{est2} and completes the proof of the proposition.
\end{proof}

In summary,  for each $f\in L^2_N$ we have the decomposition
\begin{align*}
e^{\mathcal{A}[\phi]t}f(x)&=e^{\mathcal{A}[\phi]t}\mathcal{P}_{0,N}f(x)+\phi'(x)s_{p,N}(t)f(x)\\
&\qquad+\widetilde{S}_{c,N}(t)f(x)+\widetilde{S}_{lf,N}(t)f(x)+S_{hf,N}(t)f(x),
\end{align*}
where $\mathcal{P}_{0,N}$ is the projection of $L^2_N$ onto the $NT$-periodic kernel of $\mathcal{A}[\phi]$ defined in \eqref{e:proj},
the operators $s_{p,N}(t)$ and $\widetilde{S}_{lf,N}(t)$ are defined as in \eqref{e:sp}, and the operators $\widetilde{S}_{c,N}(t)$ and $S_{hf,N}(t)$ are defined as in \eqref{e:slf} and \eqref{e:shf}, respectively.  Recalling the estimates \eqref{e:sc_bd}-\eqref{e:sp_bd}, Proposition \ref{P:Uniform_Bounds} implies that for all $N\in\NM$ we have
\[
\left\|s_{p,N}(t)f\right\|_{L^2_N}\lesssim (1+t)^{-1/4}\|f\|_{L^1_N},~~{\rm and}~~\left\|\widetilde{S}_{c,N}(t)f\right\|_{L^2_N}\lesssim (1+t)^{-3/4}\|f\|_{L^1_N},
\]
valid for all $t\geq 0$.  Together with the exponential decay estimates \eqref{e:slf_bd} and \eqref{e:shf_bd} on $\widetilde{S}_{lf,N}(t)$ and $S_{hf,N}(t)$, respectively,
and defining for each $f\in L^2_N$ the function\footnote{Clearly $\gamma_N$ is an $NT$-periodic function of $x$.}
\[
\gamma_N(x,t):=\frac{1}{N}\left<\widetilde{\Phi}_0,f\right>_{L^2_N}+s_{p,N}(t)f(x)
\]
and noting that
\[
\frac{1}{N}\left<\widetilde{\Phi}_0,f\right>_{L^2_N}=\frac{\left<\phi',\mathcal{P}_{0,N}f\right>_{L^2(0,T)}}{\|\phi'\|^2_{L^2(0,T)}}
\]
this completes the proof of Theorem \ref{T:sub_main}.

\begin{remark}
As mentioned near the end of the introduction, the methodology used in the above proof is very general and applies more generally to linear operators
$\mathcal{A}[\phi]$ with $T$-periodic coefficients that satisfy conditions (i)-(iii) listed in the discussion after Theorem \ref{t:stab}.
The resolvent bound in (iii) implies exponential decay of high-frequency modes, while the diffusive spectral stability condition
yields polynomial decay rates on the critical modes corresponding to spectrum near $(\lambda,\xi)=(0,0)$.  This work thus sets forth
a general methodology for establishing uniform decay rates on subharmonic perturbations of diffusively spectrally stable periodic waves in a large class
of evolution equations.
\end{remark}

Finally, we note that exponential decay rates on the semigroup $e^{\mathcal{A}[\phi]t}$ acting on $L^2_N$ may also be obtained from the above analysis.
Indeed, note that for each fixed $N\geq 2$ we have the bounds\footnote{Note that since $\Omega_1=\{0\}$, in the case $N=1$ 
we clearly have $\sum_{\xi\in\Omega_1\setminus\{0\}}e^{-2d\xi^2 t}=0$ and 
$\sum_{\xi\in\Omega_1}\xi^2 e^{-2d\xi^2t}=0$
for all $t\geq 0$.}
\[
\frac{1}{N}\sum_{\xi\in\Omega_N\setminus\{0\}}e^{-2d\xi^2 t}\leq \left(1-\frac{1}{N}\right)e^{-2d(\Delta \xi)^2t}
\]
and, similarly,
\[
\frac{1}{N}\sum_{\xi\in\Omega_N}\xi^2e^{-2d\xi^2 t}\leq \left(1-\frac{1}{N}\right)\left(\frac{\pi}{T}\right)^2 e^{-2d(\Delta\xi)^2 t},
\]
where here\footnote{Note for large $N$ that the size of the spectral gap $\max_{\xi\in\Omega_N\setminus\{0\}}\Re(\lambda_c(\xi))$
is $\mathcal{O}((\Delta\xi)^2)$.}
\[
\Delta\xi=\frac{2\pi}{NT}.
\]
In particular, from \eqref{e:sp}-\eqref{e:sp_bd} and the decompositions \eqref{e:shf}-\eqref{e:slf_bd} it immediately follows that there exists a constant $C>0$ 
such that for each $N\in\NM$ and $f\in L^2_N$ we have the exponential decay bound\footnote{In fact, we can improve \eqref{e:exp_decay} so that it is a bound from $L^2_N\to L^2_N$.}   
\begin{equation}\label{e:exp_decay}
\left\|e^{\mathcal{A}[\phi]t}\left(1-\mathcal{P}_{0,N}\right)f\right\|_{L^2_N}\leq Ce^{-\delta t}\|f\|_{L^1_N\cap L^2_N}, \qquad \delta := \min\{\eta, d(\Delta\xi)^2\},
\end{equation}
with $\eta>0$ as in \eqref{e:slf_bd},
recovering, at the linear level, the exponential stability result from \cite{SS19}: see Theorem \ref{T:SS} in the introduction.  In fact, the above observation extends the exponential
bound result used in \cite{SS19} since the constant $C>0$ above does not depend on $N$: see also Remark \ref{r:subharmonic}(ii) where the constant depends on $N$.  
Observe, however, that the exponential rate of decay exhibited above still tends to zero as $N\to\infty$.

\section{Sharpness of Localized Theory}\label{S:sharp}

While the polynomial decay rates established in Proposition \ref{P:Uniform_Bounds} provide upper bounds on the uniform decay rates of subharmonic perturbations of a given
diffusively stable, $T$-periodic standing wave solution of the LLE \eqref{e:LLE}, it is not a-priori clear that such decay rates are sharp, i.e. if it is possible that subharmonic perturbations
actually experience even faster uniform rates of decay.  After all, for each fixed $N\in\NM$ we already know such perturbations exhibit exponential decay (up to a null translational mode).
The next result shows that, in fact, the localized decay rates provided in Theorem \ref{t:stab} provide \emph{sharp uniform decay rates} for subharmonic perturbations.

\begin{proposition}\label{P:sharp}
There exists a constant $C>0$ such that for all $N\in\NM$ and $t>0$ we have
\begin{equation}\label{e:sub_sharp1}
\left|\frac{2\pi}{NT}\sum_{\xi_j\in\Omega_N\setminus\{0\}} e^{-2d\xi_j^2t} - \int_{-\pi/T}^{\pi/T} e^{-2d\xi^2t}d\xi\right|\leq \frac{C}{N}
\end{equation}
and
\begin{equation}\label{e:sub_sharp2}
\left|\frac{2\pi}{NT}\sum_{\xi_j\in\Omega_N} \xi_j^2e^{-2d\xi_j^2t} - \int_{-\pi/T}^{\pi/T}\xi^2 e^{-2d\xi^2t}d\xi\right|\leq \frac{C}{N(1+t)}.
\end{equation}
\end{proposition}

\begin{remark}
Note that since the terms $e^{-2d\xi^2 t}$ are exponentially small outside a small ball near $\xi=0$ for large enough time, the domain of integration above could
be replaced with $(-\infty,\infty)$, giving an even more direct connection to the bounds \eqref{e:sc_bd}-\eqref{e:sp_bd} in the subharmonic
case and the bounds \eqref{e:loc_sc_bd} and \eqref{e:loc_sp_bd} in the localized case.
\end{remark}

The proof of Proposition \ref{P:sharp} is based on a careful rescaling and Riemann sum argument and is included in the appendix.  As a consequence we see convergence, say in
$L^\infty_t(0,\infty)$, as $N\to\infty$ of the discrete sum to the integral associated to the localized theory.  This proves that the \emph{uniform} (in $N$) bounds on the operators $s_{p,N}(t)$
and $\widetilde{S}_{c,N}(t)$ provided in Theorem \ref{T:sub_main} are \emph{sharp}.  Of course, as mentioned above and several times in the manuscript, while each subharmonic
perturbation exhibits exponential decay in time, it follows that the localized theory precisely describes subharmonic decay rates which are uniform in $N$.
Furthermore, as shown in the proof, the estimates \eqref{e:sub_sharp1}-\eqref{e:sub_sharp2} show explicitly that for a fixed $N$ the sums are good approximations of the associated
integral on time scales at most\footnote{This behavior on times scales at most $\mathcal{O}(N^2)$, as well as the exponential behavior for larger times,
is readily observed numerically.} $t=\mathcal{O}(N^2)$. 
For even larger times, the exponential nature of the summands dominate and the sum decays monotonically to zero
at an exponential rate.

Naturally, it is interesting to try to recover the localized theory from the subharmonic theory in the limit as $N\to\infty$.  For example, using the boundedness
of $\widetilde{\Phi}_0$ we note there exists a constant $C>0$ independent of $N\in\NM$ such
that\footnote{Note if we use the Cauchy-Schwartz inequality to control the quantity \eqref{e:convergence1} by $L^2_N$, decay in $N$ is not observed
due to the $T$-periodicity of $\widetilde{\Phi}_0$.}
\begin{equation}\label{e:convergence1}
 \left|\left<\widetilde{\Phi}_0,f\right>_{L^2_N}\right|	\leq C\|f\|_{L^1_N}.
\end{equation}
and hence the triangle inequality implies that 
\[
\left\|e^{\mathcal{A}[\phi]t}f\right\|_{L^2_N}\leq C\left(\frac{1}{\sqrt{N}}+(1+t)^{-1/4}\right)\|f\|_{L^1_N\cap L^2_N}.
\]
Consequently, if 
$\{f_N\}_{N=1}^\infty$ is a sequence of functions with $f_N\in L^2_N\cap L^1_N$ for each $N$ and if there exists a $v\in L^2(\RM)\cap L^1(\RM)$ such
that $f_N\to v$, say in $L^1_{\rm loc}(\RM)$, then formally taking $N\to\infty$ allows us to (again formally) recover the stability result to the localized perturbation $v$ established in Theorem \ref{t:stab}.
Furthermore, the estimate \eqref{e:convergence1} suggests that we should have (in some appropriate sense)
\[
\lim_{N\to\infty}\gamma_N(x,t;f_N)=\gamma(x,t;v),
\]
where $\gamma$ is the modulation function associated to $v\in L^1(\RM)\cap L^2(\RM)$; that is, we should have convergence of the associated subharmonic
and localized modulation functions.  Of course, to make 
rigorous sense of these (and other) limiting results one must deal with the fact  that a sequence of $NT$-periodic functions can only converge to a function
in $L^2(\RM)$ at best locally in space.  Establishing such a convergence results rigorously is left as an open problem.

\section{Towards Nonlinear Stability}\label{S:nonlinear}

Our main results, Theorem \ref{T:sub_main}, Theorem \ref{t:stab} and Corollary \ref{c:behavior} give insight into the asymptotic stability and long-time modulational dynamics
near a diffusively spectrally stable periodic stationary solution of the Lugiato-Lefever equation \eqref{e:LLE} to both subharmonic and localized perturbations.  
Specifically, we note all these results are established at the linear level and that an analogous theory describing 
the full nonlinear dynamics of \eqref{e:LLE} is currently an open problem.  In this section, we describe here an approach to this problem which has been useful in
the related context of dissipative conservation or balance laws \cite{JNRZ_Invent}, as well as the unresolved difficulties involved.

To begin, suppose $\phi$ is a $T$-periodic, diffusively spectrally stable stationary periodic solution of \eqref{e:LLE} and consider  
\eqref{e:LLE} equipped with the initial condition
\begin{equation}\label{IVP}
u(x,0)=\phi(x)+v_0(x)
\end{equation}
for some sufficiently smooth and small initial perturbation $v_0$.  For the sake of clarity, let us assume the initial perturbation is localized, i.e. that $v_0\in L^2(\RM)$.
So long as it exists, Theorem \ref{t:stab} suggests decomposing the solution $u(x,t)$ of \eqref{IVP} as
\[
u(x,t)=\phi(x+\gamma(x,t))+v(x,t)
\]
where here $\gamma$ is some appropriate spatial modulation to be specified\footnote{Thanks to the invariance of \eqref{e:LLE} with respect to spatial translations,
we may assume that $\gamma(x,0)=0$.} as needed in the analysis.  Substituting this ansatz into \eqref{IVP} implies the perturbation $v$ satisfies an evolution equation
of the form
\begin{equation}\label{e:pert_per}
\left(\partial_t-\mathcal{A}[\phi]\right)\left(v-\gamma\phi'\right)=\mathcal{N}(v,v_x,v_{xx},\gamma_t,\gamma_x,\gamma_{xx}),
\end{equation}
where $\mathcal{N}$ consists of nonlinear terms in its arguments and their derivatives.  Using Duhamel, the above nonlinear evolution equation is equivalent
to the following (implicit) integral equation
\begin{equation}\label{duhamel_final}
v(x,t)-\gamma(x,t)\phi'(x)=e^{\mathcal{A}[\phi]t}v(x,0)+\int_0^t e^{\mathcal{A}[\phi](t-s)}\mathcal{N}(v,v_x,v_{xx},\gamma_t,\gamma_x,\gamma_{xx})(x,s)ds.
\end{equation}
Since the linear estimates in Theorem \ref{t:stab} imply the linearized solution operator can be decomposed as
\[
e^{\mathcal{A}[\phi]t}f(x)=\phi'(x)\left(s_p(t)f\right)(x)+\widetilde{S}(t)f(x)
\]
where $s_p$ is defined as in \eqref{phase} and $\|\widetilde{S}(t)f\|_{L^2(\RM)}\lesssim (1+t)^{-3/4}\|f\|_{L^1(\RM)\cap L^2(\RM)}$, it is natural to choose the modulation
$\gamma$ above to exactly cancel all the $s_p$ contributions on the right hand side of \eqref{duhamel_final}, thus leaving a coupled set of integral
equations for $v$ and $\gamma$ that one might hope to solve and study via contraction.  

At this point in the argument, things begin to break down.  In particular, observe that control over the source term in \eqref{duhamel_final} in $L^2(\RM)$ would require
control over (at least) $v$ in $H^2(\RM)$, corresponding to a loss of derivatives in $v$ and a-priori leaving little hope of studying \eqref{duhamel_final} via
iteration.  In some cases, however, such a loss of derivatives at the linear level can be compensated by appropriately strong damping effects of the governing evolution equation: see, for example,
\cite{JNRZ_Invent}.  However, in the case of the LLE \eqref{e:LLE} the damping actually corresponds to the \emph{lowest-order derivative\footnote{In fact, a zeroth order term.}}
which elementary calculations show negates the  ``nonlinear damping" technique leveraged in \cite{JNRZ_Invent} to regain derivatives at the nonlinear level.

In summary, when attempting to upgrade our linear results (either in the localized or subharmonic setting) to a nonlinear theory one is confronted with an iteration scheme which
a-priori loses regularity.  Dealing with this loss of regularity would be an exciting direction for future research.  

\begin{remark}
It is interesting to note that in the case of subharmonic perturbations, if one fixes $N\in\NM$ and uses the exponential decay bounds on $e^{\mathcal{A}[\phi]t}$
on $L^2_N$, as in the recent work \cite{SS19}, then there is no need for the modulation function $\gamma$ to have $x$-dependence: see, for example,
Theorem \ref{T:SS}.  In this case, one can show that the nonlinear term in \eqref{e:pert_per} reduces to simply $\mathcal{N}(v,\gamma_t)$ and hence, in this case,
there would be no loss of derivatives.  This is precisely the strategy behind the proof of Theorem \ref{T:SS} presented in \cite{SS19}.  As described throughout this manuscript,
however, the exponential (nonlinear) decay rate is not uniform in $N$, and our analysis suggests that to obtain such uniform decay rates on subharmonic
perturbations one must work with space-time dependent modulation functions, which (as detailed above) leads to a loss of regularity in our iteration scheme.  
A similar situation occurs in the seminal work \cite{Kap97}, where the author considers the nonlinear transverse stability of planar traveling wave solutions to reaction diffusion systems.
In that case, the lack of spectral gap (due to the transverse modes) is compensated by the introduction of a modulation function which depends on time and the transverse spatial
directions.  However, since the background wave is constant in the transverse spatial directions, this leads to the a set of nonlinear perturbation equations where the perturbation 
(as in the work \cite{SS19} discussed above) enters without derivatives acting on it, and hence there is no loss of derivatives in the iteration scheme.
Establishing nonlinear results when one has the ability to regain these lost derivatives is currently under investigation by the authors.
\end{remark}

\appendix
\section{Proof of Proposition \ref{P:sharp}}

In this appendix, we present a proof of Proposition \ref{P:sharp}, which establishes that the localized rates of decay in Theorem \ref{t:stab} provide
\emph{sharp} uniform bounds on subharmonic perturbations of diffusively spectrally stable periodic standing waves of the LLE \eqref{e:LLE}.

\begin{proof}[Proof of Proposition \ref{P:sharp}]
As in the proof of Proposition \ref{P:Uniform_Bounds}, it is sufficient to establish the bounds \eqref{e:sub_sharp1}-\eqref{e:sub_sharp2} for $N\geq 2$ and $t\geq 1$.
Since the proof of the estimates \eqref{e:sub_sharp1} and \eqref{e:sub_sharp2} for $N\geq 2$ and $t\geq 1$ 
follow the same basic structure, we present a detailed proof of \eqref{e:sub_sharp2} and then describe the modifications needed to establish \eqref{e:sub_sharp1}.

To begin our proof of \eqref{e:sub_sharp2}, fix $N\in\NM$ with $N\geq 2$ and define the function
\begin{equation}\label{e:fN}
F_N(t):=t^{3/2}\sum_{\xi_j\in\Omega_N}\xi_j^2e^{-2d\xi_j^2 t}\Delta\xi_j,~~t\geq 1,
\end{equation}
where here for each $j\in\NM$ we set $\Delta\xi_j=\frac{2\pi}{NT}$.  
We now fix $t>0$ and use a discrete change of variables by setting, for each $\xi_j\in\Omega_N$,
\begin{equation}\label{e:rescale1}
z_j = \xi_j \sqrt{t}
\end{equation}
noting that, in particular, the set $\Omega_{N,t} := \{z_j\}$ is a discretization of $[-\pi\sqrt{t}/T,\pi\sqrt{t}/T]$ with 
\[
\Delta z_j = \Delta\xi_j\sqrt{t} = 2\pi\sqrt{t}/(NT).
\]
Further, in terms of the new variables we have 
\[
F_N(t) = \sum_{z_j\in\Omega_{N,t}} z_j^2e^{-2dz_j^2}\Delta z_j,
\]
and hence $F_N$ is a 
Riemann sum approximation for the integral
\[
\int_{-\pi\sqrt{t}/T}^{\pi\sqrt{t}/T} H(z)dz,~~{\rm where}~~H(z):=z^2 e^{-2dz^2}.
\]
Our strategy is to treat $F_N$ as either a left or right endpoint Riemann sum approximation of the integral, and our choice will be dictated by the intervals where $H$ is increasing or decreasing.  
From this, we will establish the bound\footnote{Note this clearly shows that the finite sum is a good approximation of the associated integral on time scales at most $t=\mathcal{O}(N^2)$.}
\begin{equation}\label{e:Rsum_bound}
\left|F_N(t) - \int_{-\pi\sqrt{t}/T}^{\pi\sqrt{t}/T} H(z)dz\right| \leq \frac{C\sqrt{t}}{N},
\end{equation}
which (by undoing the above rescaling), is clearly equivalent to the bound \eqref{e:sub_sharp2}.

Now, to establish \eqref{e:Rsum_bound} we first note by symmetry of the function $H$ that it suffices to study the sum defining $F_N$ over only the non-negative $z_j$, i.e. to study the function
\[
h_N(t) := \sum_{j=0}^{m} z_j^2e^{-2dz_j^2}\Delta z_j = \sum_{j=0}^m H(z_j)\Delta z_j, \quad 
m := \max\{j\mid z_j\in\Omega_{N,t}\} = \begin{cases}
\frac{N}{2}-1, & N ~\text{even}\\
\frac{N-1}{2}, & N ~\text{odd}
\end{cases}
\]
Observe that $H$ is monotonically increasing for $z\in (0,R)$, where $R=(2d)^{-1/2}$, and is monotonically decreasing for $z\in (R,\infty)$.  
In particular, for each fixed $t>0$, we have that either $z_m=z_m(t)<R$ or there exists $\ell=\ell(t)\in\{0,\ldots,m-1\}$ such that $R\in(z_\ell, z_{\ell+1}]$:
see Figure \ref{F:sum}.

\begin{figure}[t]
\begin{center}
(a)\includegraphics[scale=0.45]{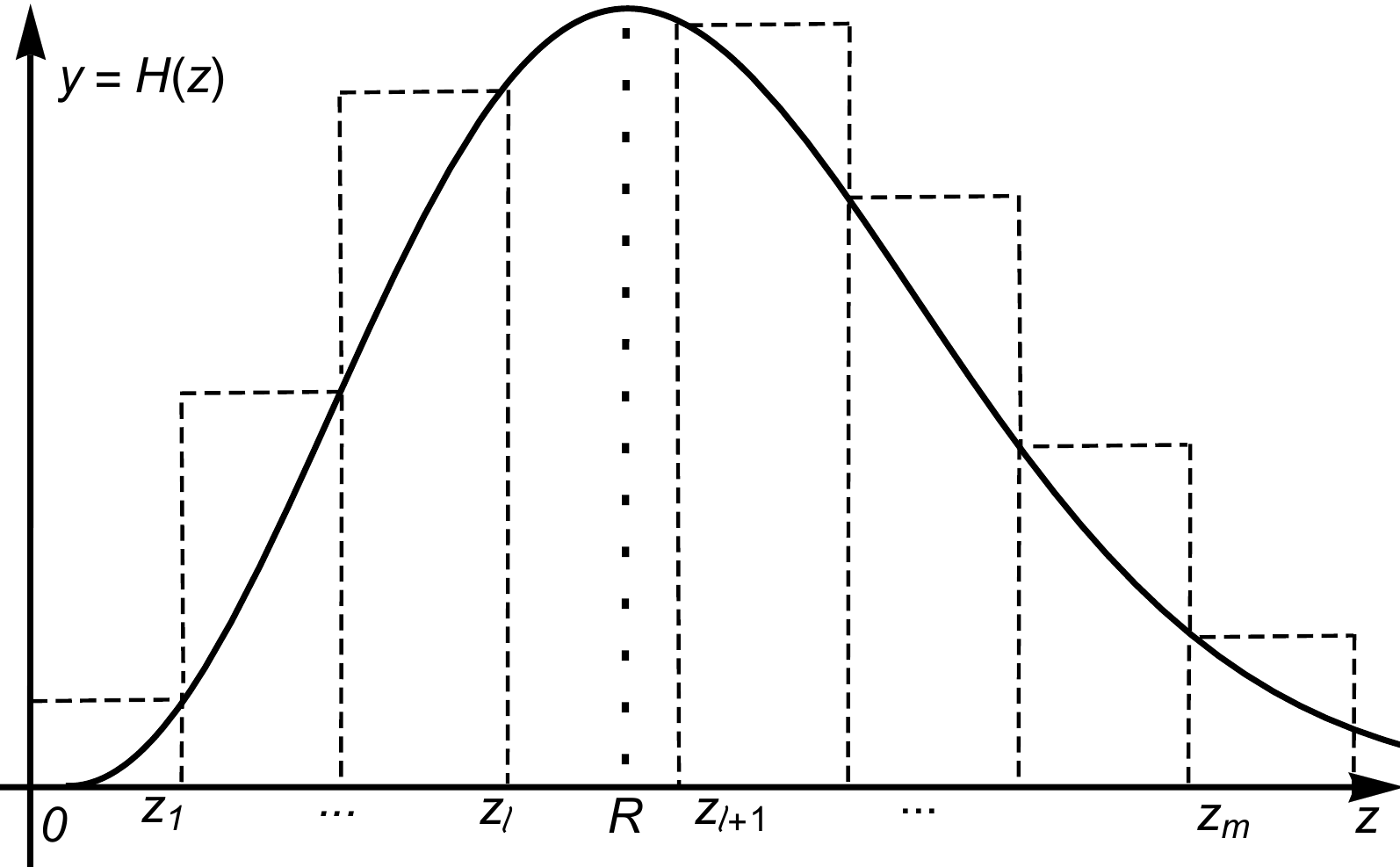}\quad (b)\includegraphics[scale=0.45]{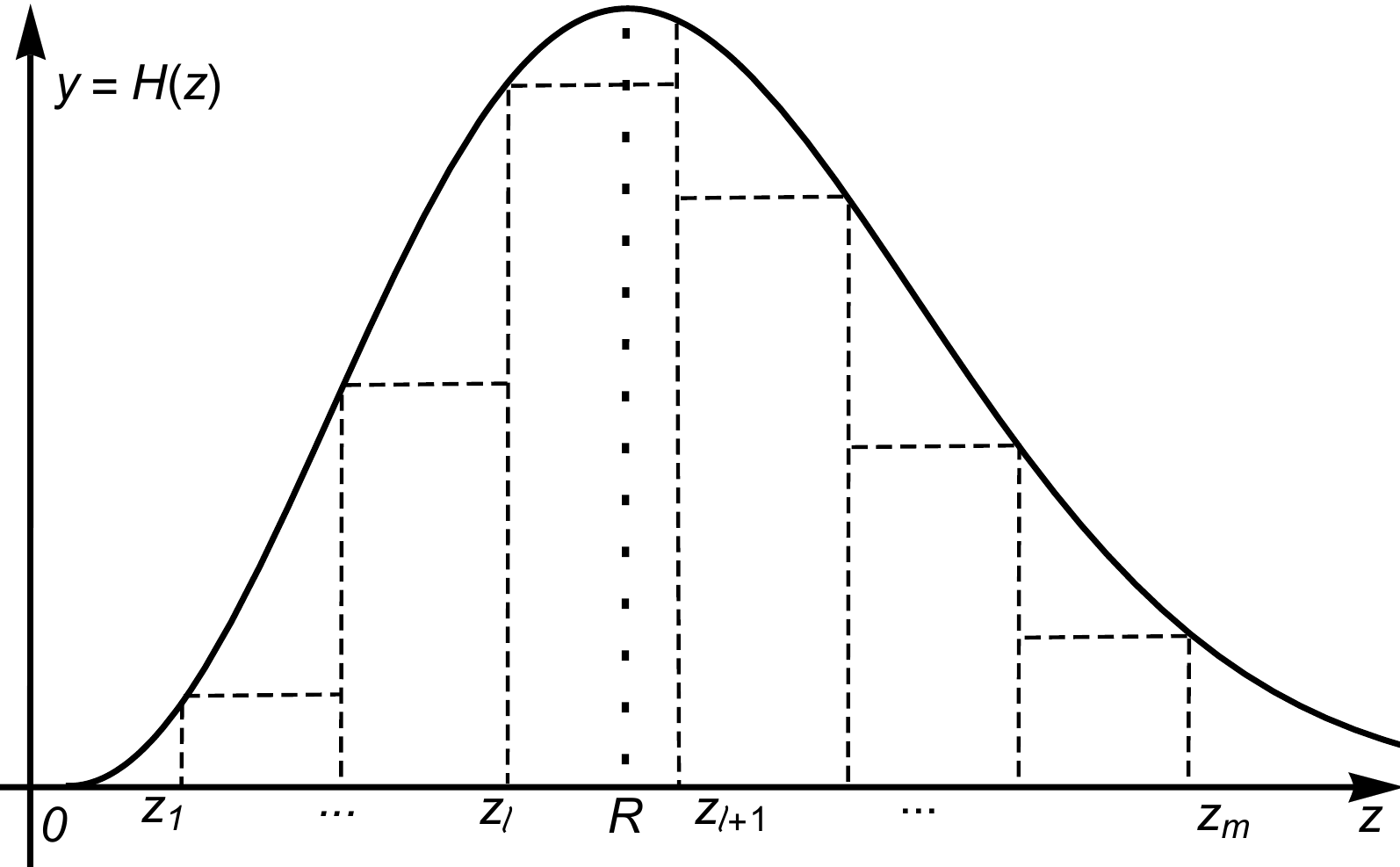}
\caption{{(a) Treating $h_N$ as an over approximation of the area under $H$.  (b) Treating $h_N$ as an under approximation. Notice that (a) misses the interval $[z_\ell, z_{\ell+1}]$.  Similarly (b) has a node, in this case $z_{\ell+1}$, that cannot produce an under approximation on the interval from $[z_\ell, z_{\ell+1}]$.  This indicates that this interval must be treated with some care.} }\label{F:sum}
\end{center}
\end{figure}

Note that if $\ell=0$ then $z_j>R$ for each $j\in\{1,2,\ldots,m\}$, implying that each such $z_j$ lies in the monotonically decreasing
tail of $H$.  This case happens when $z_1>R$, i.e. when $t>C_0 N^2$ where $C_0 = T^2/(8d \pi^2)$.
Noting, as in Figure \ref{F:sum}, that the $z_\ell$ and $z_{\ell+1}$ terms need to be handled differently, we decompose $h_N$ as
\[
h_N(t) = \sum_{j=0}^{1} H(z_j)\Delta z_j + \sum_{j=2}^m H(z_j)\Delta z_j.
\]
On the opposite of this extreme,
observe that if $\ell=m-1$ then $z_j<R$ for all $j\in\{0,1,\ldots,m-1\}$, implying that each such $z_j$  lies in the monotonically increasing
portion of $H$, while $z_m \geq R$, leading to a similar decomposition to that above.  Further, we only treat $z_m$ differently when $z_m<R$, which occurs for some (fixed) bounded interval of time.  (Although, we note that, strictly speaking, $z_m$ does not have to be treated differently when $z_m + \Delta z_m < R$.)  

In our analysis below, we will only consider the case when there are some $z_j$ in both the monotonically increasing and the monotonically decreasing
components of $H$.  That is, we restrict to the case when 
\begin{equation}\label{e:R_condition}
R\in(z_\ell, z_{\ell+1}], ~~{\rm for~some}~~ \ell\in\{1,\ldots,m-2\},
\end{equation}
noting the other more extreme cases described above will follow in a more straightforward way.
In the case when \eqref{e:R_condition} holds, we decompose $h_N$ into thee components via
\[
h_N(t) = \sum_{j=0}^{\ell-1} H(z_j)\Delta z_j + \sum_{j=\ell}^{\ell+1} H(z_j)\Delta z_j + \sum_{j=\ell+2}^m H(z_j)\Delta z_j.
\]
Note the first and last terms represent the sums over the $j$ where $z_{j+1}<R$ 
and where $z_{j-1}>R$, respectively, while the middle 
term represents the sum over the $z_j$ nearest to the absolute maximum of $H$.

Because $H$ is increasing on 
$(0,R)$ it follows that
\[
H(z_j)\Delta z_j \leq \int_{z_j}^{z_{j+1}} H(z)dz \leq 
H(z_{j+1})\Delta z_{j+1}~~{\rm for}~~j=0,1,\ldots,\ell-1,
\]
and hence, recalling that $H(z_0)=0$, we have\footnote{Technically, the above bound gives $\int_0^{z_{\ell-1}} H(z)dz \leq \sum_{j=1}^{\ell-1} H(z_j)\Delta z_j$,.  Since $H(z_0)=0$, however,
this is equivalent to the stated lower estimate.}
\[
\int_0^{z_{\ell-1}} H(z)dz \leq \sum_{j=0}^{\ell-1} H(z_j)\Delta z_j \leq \int_0^{z_\ell} H(z)dz.
\]
Similarly, since $H$ is decreasing on $(R,\infty)$ 
it follows that
\[
H(z_j)\Delta z_{j} \leq \int_{z_{j-1}}^{z_j} H(z)dz \leq H(z_{j-1})\Delta z_{j-1}~~{\rm for}~~j=\ell+2,\ldots,m
\]
which, as above, yields
\[
\int_{z_{\ell+2}}^{z_{m}+\Delta z_m} H(z)dz \leq \sum_{j=\ell+2}^m H(z_j)\Delta z_j \leq \int_{z_{\ell+1}}^{z_m} H(z)dz.
\]
Together, this gives
\begin{equation}\label{bd1}
\int_0^{z_m+\Delta z_m}H(z)dz-\int_{z_{\ell-1}}^{z_{\ell+2}}H(z)dz\leq 
h_N(t)-\sum_{j=\ell}^{\ell+1}H(z_j)\Delta z_j
\leq\int_0^{z_m}H(z)dz-\int_{z_\ell}^{z_{\ell+1}}H(z)dz.
\end{equation}

It now remains to bound the sum $\sum_{j=\ell}^{\ell+1} H(z_j)\Delta z_j$.  
To this end, recall that $H$ has exactly one critical point (a global maximum at $z=R$) on $(z_\ell, z_{\ell+1}]$.  Consequently, since the minimum of $H$
on $[z_\ell,z_{\ell+1}]$ must occur at one of the endpoints, we have
\[
\min_{j=\ell,\ell+1}H(z_j)\Delta z_j \leq \int_{z_\ell}^{z_{\ell+1}} H(z)dz.
\]
while the other endpoint is at worst the supremum of $H$, yielding
\[
\max_{j=\ell,\ell+1}H(z_j)\Delta z_j\leq H(R)\left(2\pi\sqrt{t}/(NT)\right)= 
C\sqrt{t}/N, 
\]
where $C = \pi/(dTe)$.    The previous two bounds together yield 
\[
\sum_{j=\ell}^{\ell+1}H(z_j)\Delta z_j\leq\int_{z_\ell}^{z_{\ell+1}}H(z)dz+C\sqrt{t}/N.
\]
On the other hand, recalling that $H$ is increasing on $(z_{\ell-1},z_\ell)$ and is decreasing on $(z_{\ell+1},z_{\ell+2})$ gives the lower bound
\[
H(z_\ell)\Delta z_\ell \geq \int_{z_{\ell-1}}^{z_\ell}H(z)dz, ~~{\rm and}~~ H(z_{\ell+1})\Delta z_{\ell+1} \geq \int_{z_{\ell+1}}^{z_{\ell+2}}H(z)dz.
\]
Together with the estimate \eqref{bd1}, it follows that
\begin{equation}\label{bd2}
\int_{0}^{z_m+\Delta z_m} H(z)dz - \int_{z_{\ell}}^{z_{\ell+1}} H(z)dz \leq h_N(t) \leq \int_{0}^{z_m} H(z)dz + C\sqrt{t}/N.
\end{equation}
Recalling again that $H$ attains its global maximum at $z=R\in[z_\ell,z_{\ell+1}]$, we clearly have
\[
\int_{z_{\ell}}^{z_{\ell+1}} H(z)dz \leq 
H(R)\left(2\pi\sqrt{t}/(NT)\right)=C\sqrt{t}/N.
\]
and hence, from \eqref{bd2}, we obtain
\begin{equation}\label{bd3}
\int_{0}^{z_m+\Delta z_m} H(z)dz - C\sqrt{t}/N \leq h_N(t) \leq \int_{0}^{z_m} H(z)dz + C\sqrt{t}/N.
\end{equation}
By the even symmetry of $H$ and the structure of the sets $\Omega_N$, we find that
\begin{equation}\label{bd4}
\int_{z_{-\widetilde{m}}-\Delta z_{-\widetilde{m}}}^0 H(z) dz - C\sqrt{t}/N \leq \sum_{j=-\widetilde{m}}^{-1} H(z_j) \Delta z_j \leq \int_{z_{-\widetilde{m}}}^0 H(z) dz + C\sqrt{t}/N,
\end{equation}
where
\[
\widetilde{m} := |\min\{j\mid z_j\in\Omega_{N,t}\}| = \begin{cases}
\frac{N}{2}, & N ~\text{even}\\
\frac{N-1}{2}, & N ~\text{odd}
\end{cases}.
\]
Therefore, combining \eqref{bd3} and \eqref{bd4} we obtain the bound
\[
\int_{z_{-\widetilde{m}}-\Delta z_{-\widetilde{m}}}^{z_{m}+\Delta z_m} H(z) dz - 2C\sqrt{t}/N \leq F_N(t) \leq \int_{z_{-\widetilde{m}}}^{z_m} H(z) dz + 2C\sqrt{t}/N.
\]
Noting that 
\[
\int_{z_{-\widetilde{m}}}^{z_m} H(z) dz \leq \int_{-\pi\sqrt{t}/T}^{\pi\sqrt{t}/T} H(z) dz \leq \int_{z_{-\widetilde{m}}-\Delta z_{-\widetilde{m}}}^{z_{m}+\Delta z_m} H(z) dz,
\]
we obtain the bound \eqref{e:Rsum_bound}, as desired.  By undoing the rescaling \eqref{e:rescale1}, this establishes the estimate \eqref{e:sub_sharp2} in Proposition
\ref{P:sharp}.

Finally, the proof of the  estimate \eqref{e:sub_sharp1} follows by a similar, yet simpler, analysis.
Indeed, for all $N\geq 2$ and $t\geq 1$ we define
\[
\widetilde{F}_N(t) := t^{1/2}\sum_{\xi_j\in\Omega_N\setminus\{0\}}e^{-2d\xi_j^2 t}\Delta\xi_j 
\]
and note by performing the same discrete change of variables in \eqref{e:rescale1} as before we have
\[
\widetilde{F}_N(t) = \sum_{z_j\in\Omega_{N,t}\setminus\{0\}} e^{-2d z_j^2}\Delta z_j.
\]
Since the function $\widetilde{H}(z)=e^{-2dz^2}$ is even and strictly decreasing for $z>0$, carrying out the same monotonicity argument as above we find that
\[
\int_{z_{-\widetilde{m}}-\Delta z_{-\widetilde{m}}}^{z_{m}+\Delta z_m} \widetilde{H}(z) dz - \int_{z_{-1}}^{z_1}\widetilde{H}(z)dz \leq \widetilde{F}_N(t) 
	\leq \int_{z_{-\widetilde{m}}}^{z_m} \widetilde{H}(z) dz.
\]
In particular, using the elementary bound
\[
\int_{z_{-1}}^{z_1}\widetilde{H}(z)dz\leq \widetilde{H}(0)(z_1-z_{-1})=\frac{4\pi\sqrt{t}}{NT}
\]
and noting that
\[
\int_{z_{-\widetilde{m}}}^{z_m} \widetilde{H}(z) dz \leq \int_{-\pi\sqrt{t}/T}^{\pi\sqrt{t}/T} \widetilde{H}(z) dz 
	\leq \int_{z_{-\widetilde{m}}-\Delta z_{-\widetilde{m}}}^{z_{m}+\Delta z_m} \widetilde{H}(z) dz,
\]
we obtain the estimate
\begin{equation}\label{e:Rsum_bound2}
\left|\widetilde{F}_N(t) - \int_{-\pi\sqrt{t}/T}^{\pi\sqrt{t}/T} e^{-2dz^2}dz\right| \leq \frac{4\pi\sqrt{t}}{NT}
\end{equation}
valid for all $N\geq 2$ and $t\geq 1$.  Undoing the rescaling \eqref{e:rescale1}, this completes the proof of \eqref{est1}.
\end{proof}

\begin{remark}
The reader may wonder why we used a rescaling and monotonicity argument above to establish \eqref{e:Rsum_bound} and \eqref{e:Rsum_bound2}, as opposed
to a more direct Mean Value Theorem argument.  Essentially, this is because the Mean Value Theorem gives us the cruder bound
\[
\left|F_N(t) - \int_{-\pi\sqrt{t}/T}^{\pi\sqrt{t}/T} H(z)dz\right| \leq \frac{Mt}{N}.
\]
\end{remark}

Finally we note that the uniform bounds in Proposition \ref{P:Uniform_Bounds} may be recovered from the above analysis.  For example, 
note that the function $F_N$ defined in \eqref{e:fN} satisfies the differential inequality
\[
F_N'(t) \leq t^{1/2}\left(\frac{3}{2} - \frac{8d\pi^2 t}{N^2 T^2}\right) \sum_{\xi_j\in\Omega_N}\xi_j^2e^{-2d\xi_j^2 t}\Delta\xi_j
			=t^{-1}\left(\frac{3}{2} - \frac{8d\pi^2 t}{N^2 T^2}\right) F_N(t)
\]
and hence exhibits exponential decay to zero for time scales larger than $\mathcal{O}(N^2)$.  Since \eqref{e:Rsum_bound} shows that $F_N$ is uniformly
bounded on time scales up to $\mathcal{O}(N^2)$ it follows that $F_N(t)$ is uniformly bounded (in $N$) for all $t\geq 1$, which
establishes \eqref{est2} in Proposition \ref{P:Uniform_Bounds}.  A similar differential inequality argument applied to $\widetilde{F}_N(t)$
establishes the uniform estimate \eqref{est1}.

\bibliographystyle{abbrv}
\bibliography{LLE}

\end{document}